\documentclass[11pt]{article}
\usepackage[T1]{fontenc}
\usepackage{lmodern}
\usepackage[colorlinks=true,linkcolor=RoyalBlue,urlcolor=RoyalBlue,citecolor=red]{hyperref}
\usepackage{amsmath,amsthm,amssymb,amsfonts}
\usepackage{enumerate}
\usepackage{epsfig}
\usepackage[dvipsnames,svgnames,table]{xcolor}
\usepackage{pgf,tikz}
\usetikzlibrary{calc}
\usepackage{fullpage}

\usepackage{booktabs}
\usepackage[nameinlink]{cleveref}
\usepackage{enumitem}
\usepackage{blkarray}
\usepackage{tikz-cd}

\usepackage{thmtools, thm-restate}

\newtheorem{theorem}{Theorem}[section]
\newtheorem{proposition}[theorem]{Proposition}
\newtheorem{lemma}[theorem]{Lemma}
\newtheorem{corollary}[theorem]{Corollary}

\newtheorem{claim}[theorem]{Claim}

\theoremstyle{definition}
\newtheorem{definition}[theorem]{Definition}
\newtheorem{example}[theorem]{Example}

\theoremstyle{remark}

\newenvironment{claimproof}[1][\proofname]
	{
		\proof[#1]%
			
	}
	{
		\endproof
	}

\newcommand{\CC}{\mathbb{C}}
\newcommand{\RR}{\mathbb{R}}
\newcommand{\QQ}{\mathbb{Q}}

\newcommand{\FF}{\mathbb{F}}

\DeclareMathOperator{\rank}{rank}

\DeclareMathOperator{\rig}{rig}
\DeclareMathOperator{\fig}{fig}

\DeclareMathOperator{\td}{td}

\newcommand{\added}[1]{{\color{red}#1}}

\tikzstyle{vertex}=[fill=black,circle,inner sep=0pt, minimum size=4pt]
\tikzstyle{edge}=[line width=1.5pt,black]

\title{The $d$-dimensional realisation number of a rigid graph}
\author{Sean Dewar\thanks{Department of Computer Science, KU Leuven. E-mail: \texttt{sean.dewar@kuleuven.be}}  \and Anthony Nixon\thanks{School of Mathematical Sciences, Lancaster University. E-mail: \texttt{a.nixon@lancaster.ac.uk}}  \and Ben Smith\thanks{School of Mathematical Sciences, Lancaster University. E-mail: \texttt{b.smith9@lancaster.ac.uk}}}

\begin{document}
\date{}
\maketitle

\begin{abstract}
Determining the number of (complex) realisations of a rigid graph for a specific choice of edge lengths is a fundamental problem in discrete geometry. 
In this article we provide two new tools for determining realisation numbers in arbitrary dimensions:
(i) we prove that subgraph inclusion translates to realisation number divisibility; and
(ii) we provide lower bounds on realisation numbers under specific graph operations in all dimensions.
We use these methods to prove that every triangulated sphere with $n$ vertices has at least $2^{n-4}$ edge-length equivalent realisations in 3-dimensions, extending a 2-dimensional result of Jackson and Owen in the case of planar graphs.
Additionally, our tools solve a family of conjectures set by Grasegger regarding how 1-extensions, X-replacements, and V-replacements affect realisation numbers.
\end{abstract}

{\small \noindent \textbf{MSC2020:} 52C25, 05C10, 68R12, 14C17}

{\small \noindent \textbf{Keywords:} rigid graph, bar-joint framework, realisation number, complex realisations, vertex splitting, triangulation}

\section{Introduction}\label{sec:intro}

A (bar-joint) framework $(G,p)$ in $\mathbb{R}^d$ is the combination of a finite, simple graph $G=(V,E)$ and a realisation $p:V\rightarrow \mathbb{R}^d$. The framework is \emph{rigid} if the only edge-length-preserving continuous deformations of the vertices arise from isometries of $\mathbb{R}^d$, and is \emph{flexible} otherwise. 
Asimow and Roth \cite{AsimowRothI} showed that rigidity of a graph is a generic property. That is, if there exists one generic realisation $p$ evidencing that $G$ is rigid then it is certain that every generic framework $(G,q)$ is rigid.
Because of this, we say a graph is \emph{$d$-rigid} if there exists a generic rigid framework $(G,p)$ in $\mathbb{R}^d$.

A number of applications require more detailed information such as the following: given a rigid framework, how many edge-length equivalent realisations exist up to isometries of $\mathbb{R}^d$?
For a framework $(G,p)$, we denote this value by $r_d(G,p)$.
The combined results of Connelly \cite{conggr} and Gortler, Healy and Thurston \cite{gortler2010characterizing} show that, if $p$ is generic, whether $r_d(G,p)=1$ depends solely on the underlying graph. We thus say that $G$ is \emph{globally $d$-rigid} if there exists a generic realisation $p$ such that $(G,p)$ is the unique realisation of $G$ in $\mathbb{R}^d$ with the given edge lengths up to isometries.

In general, however, the value $r_d(G,p)$ is not a generic property. 
Indeed, it is not hard to construct small examples of graphs that have different numbers of realisations depending on the edge lengths chosen; see for example \Cref{fig:notgeneric}.
To get around this, we can associate a single `generic realisation number' to a graph in two different ways:
\begin{enumerate}
    \item We define the \emph{real $d$-realisation number} -- here denoted $r_d(G)$ -- to be the largest value of $r_d(G,p)$ over all generic realisations.
    \item We extend $r_d(G,p)$ to additionally count complex realisations, here denoted $c(G,p)$. The value $c(G,p)$ \emph{is} a generic property, hence we define the \emph{(complex) $d$-realisation number} -- here denoted $c_d(G)$ -- to be $c(G,p)$ for any generic $p$.
\end{enumerate}
It is immediate that $c_d(G) \geq r_d(G)$ for all graphs, and it was proven by Gortler and Thurston \cite{Gortler2014} that $r_d(G)=1$ (i.e., $G$ is globally $d$-rigid) if and only if $c_d(G)=1$.

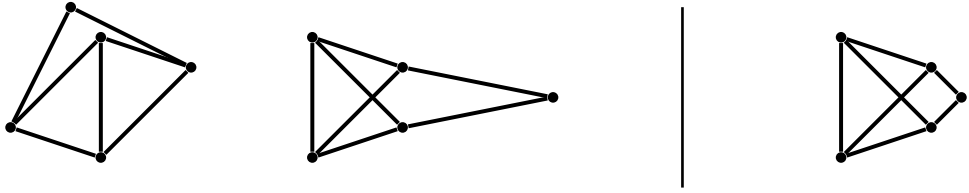
\begin{figure}[tp]
	\begin{center}
        \begin{tikzpicture}[scale=0.8]
        \begin{scope}[xshift=0]
			\node[vertex] (1) at (0,1) {};
			\node[vertex] (2) at (0,-1) {};
			
			\node[vertex] (3) at (-1.5,-0.5) {};
			\node[vertex] (4) at (1.5,0.5) {};
			\node[vertex] (5) at (-0.5, 1.5) {};
			
			\draw[edge] (1)edge(2);
			\draw[edge] (1)edge(3);
			\draw[edge] (1)edge(4);
			\draw[edge] (2)edge(3);
			\draw[edge] (2)edge(4);
			\draw[edge] (5)edge(3);
			\draw[edge] (5)edge(4);
        \end{scope}
		\begin{scope}[xshift=100]
			\node[vertex] (1) at (0,1) {};
			\node[vertex] (2) at (0,-1) {};
			
			\node[vertex] (3) at (1.5,-0.5) {};
			\node[vertex] (4) at (1.5,0.5) {};
			\node[vertex] (5) at (4,0) {};
			
			\draw[edge] (1)edge(2);
			\draw[edge] (1)edge(3);
			\draw[edge] (1)edge(4);
			\draw[edge] (2)edge(3);
			\draw[edge] (2)edge(4);
			\draw[edge] (5)edge(3);
			\draw[edge] (5)edge(4);
		\end{scope}
        \begin{scope}[xshift=275]
            \draw[line width=1pt] (0,1.5) -- (0,-1.5);
		\end{scope}
        \begin{scope}[xshift=350]
			\node[vertex] (1) at (0,1) {};
			\node[vertex] (2) at (0,-1) {};
			
			\node[vertex] (3) at (1.5,-0.5) {};
			\node[vertex] (4) at (1.5,0.5) {};
			\node[vertex] (5) at (2,0) {};
			
			\draw[edge] (1)edge(2);
			\draw[edge] (1)edge(3);
			\draw[edge] (1)edge(4);
			\draw[edge] (2)edge(3);
			\draw[edge] (2)edge(4);
			\draw[edge] (5)edge(3);
			\draw[edge] (5)edge(4);
        \end{scope}
		\end{tikzpicture}
	\end{center}
    \caption{Two different choices of edge lengths for the same graph. The left realisation gives $r_2(G,p) = 4$: the two realisations pictured plus two more via reflecting the degree 2 vertex through the line passing through its neighbours. The right realisation only gives $r_2(G,p)=2$: again, the additional realisation can be found via reflecting the degree 2 vertex.}\label{fig:notgeneric}
\end{figure}

Outside of the almost-trivial case of $d=1$,
both real and complex realisation numbers are often difficult to compute.
The complex realisation number can be viewed as the number of solutions to a system of quadratic equations, and using black-box methods such as Gr\"{o}bner bases to compute $c_d(G)$ very quickly becomes computationally untenable.
A combinatorial algorithm for computing $c_2(G)$ for minimally 2-rigid graphs is known \cite{cggkls},
which was recently extended to all 2-rigid graphs \cite{dgstw25}.
Unfortunately, no such combinatorial algorithms are known for computing $c_d(G)$ when $d \geq 3$ or $r_d(G)$ when $d \geq 2$.
Many techniques have been developed over the last two decades to provide exact values and bounds for both $c_d(G)$ and $r_d(G)$:
see for example \cite{BARTZOS2021189,BartzosUpper,BS04,JO19,Stef10}.

\subsection{Our contributions}

Our main contributions on the topic can be split into two main subtopics.

\subsubsection{Subgraph realisation numbers and divisibility}

It is easy to see that, if $H$ is a $d$-rigid spanning subgraph of a $d$-rigid graph $G$, then $c_d(G) \leq c_d(H)$.
We prove that this inequality is actually due to the divisibility of $c_d(H)$ by $c_d(G)$.

\begin{restatable}{theorem}{gcd}\label{thm:gcd}
    Let $G$ be a $d$-rigid graph on at least $d+1$ vertices.
    If $H$ is a spanning $d$-rigid subgraph of $G$,
    then $c_d(G) \mid c_d(H)$.
\end{restatable}

We also prove that, under specific circumstances, the realisation  of a $d$-rigid (but not necessarily spanning) subgraph divides the realisation number of the larger $d$-rigid graph.
This result is a generalisation of a recent result of Grasegger \cite[Lemma 3]{gras25}, who proved the statement for minimally $d$-rigid graphs.

\begin{restatable}{theorem}{isosubreal}\label{thm:isosubreal}
    Let $G$ and $H$ be $d$-rigid graphs with at least $d+1$ vertices such that $H$ is a subgraph of $G$.
    Further suppose that for some minimally $d$-rigid subgraph $\widetilde{H}$ of $H$,
    we have that $G - E(H) + E(\widetilde{H})$ is minimally $d$-rigid.
    Then $c_d(H) \mid c_d(G)$.
\end{restatable}

As an application of these results, we characterise how the realisation number changes under \emph{rigid subgraph substitution} (\Cref{cor:rigidsubgraphsub}).
We then use this to prove Conjectures 1, 3 and 4 of Grasegger \cite{gras25},
which describe conditions for which 1-extensions, X-replacements and V-replacements exactly double the realisation number (\Cref{cor:1extspecial,cor:xvreplacement}).

\subsubsection{Behaviour of realisation numbers under certain graph operations}

In \cite{JO19},
Jackson and Owen explored how different graph operations affected both the real and complex 2-realisation numbers of 2-rigid graphs.
Two important cases they explored were \emph{2-dimensional 0-extensions} and \emph{2-dimensional vertex splits}:
see \Cref{sec:operations} for descriptions of these operations.
We extend these results to all dimensions in both the real and complex case (\Cref{prop:0extreal}, \Cref{thm:vertexsplit}, \Cref{thm:vertexsplitreal}).
We additionally provide lower bounds for the \emph{$d$-dimensional spider-split} (\Cref{thm:spidersplit}, \Cref{thm:spidersplitreal}).


Jackson and Owen \cite{JO19} proved that when a graph $G$ is both minimally 2-rigid and planar, the inequality $c_2(G) \geq r_2(G) \geq 2^{|V| - 3}$ holds.
The proof combines lower bounds provided by 2-dimensional vertex splitting with a known construction of planar minimally 2-rigid graphs via 2-dimensional vertex splitting.
Utilising our results on higher-dimensional graph operations, we approach the 3-dimensional variant of this in a similar way.

Building on work of Cauchy \cite{cauchy1813}, Gluck \cite{gluck} proved that triangulated spheres are minimally 3-rigid.
Moreover, Steinitz \cite{steinitz} proved that every triangulated sphere can be constructed from $K_4$ using 3-dimensional vertex splits.
By combining these observations with our own results regarding vertex splitting, we derive a lower bound on the realisation number of a triangulated sphere in 3-space.

\begin{restatable}{theorem}{triangulatedsphere}\label{thm:triangulatedsphere}
    Let $G = (V,E)$ be a triangulated sphere.
    Then $c_3(G) \geq r_3(G) \geq 2^{|V| - 4}$.
\end{restatable}

We additionally extend \Cref{thm:triangulatedsphere} to minimally 3-rigid projective planar graphs (\Cref{thm:pp}).

\subsection{Layout of paper}

We first review notation and terminology, and collect preliminary results from rigidity theory and algebraic geometry in \Cref{sec:prelims}.
In \Cref{sec:divide}, we prove our division theorems: \Cref{thm:gcd} and \Cref{thm:isosubreal}.
In \Cref{sec:operations} we then analyse the effect of graph operations on the $d$-dimensional realisation number.
In \Cref{sec:applications}, we use the theory built up in the previous sections to derive lower bounds on the realisation number of a triangulated sphere (\Cref{thm:triangulatedsphere}), characterise the effect of rigid subgraph substitution on realisation numbers (\Cref{cor:rigidsubgraphsub}) and resolve several conjectures on which 1-extensions, X-replacements and V-replacements exactly double the realisation number (\Cref{cor:1extspecial,cor:xvreplacement}).

\section{Preliminaries}
\label{sec:prelims}

We now fix $\mathbb{F}$ to be either the field of real numbers or the field of complex numbers.
We will equip each vector space $\mathbb{F}^d$ with the quadratic form $|| (x_1,\ldots,x_d) ||^2 := \sum_{i=1}^d x_i^2$.
The notation here is intentional, as when $\mathbb{F}=\mathbb{R}$ our chosen quadratic form is exactly the square of the standard Euclidean norm.
The isometries of the quadratic space $(\mathbb{F}^d,\|\cdot\|^2)$ are exactly the affine transformations $x \mapsto Ax +b$ where $b \in \mathbb{F}^d$ and $A \in O(d,\mathbb{F})$, i.e., the group of $d \times d$ matrices with entries in $\mathbb{F}$ where $M^{\top} M = M M^{\top} = I_d$.

\subsection{Rigidity theory for real and complex frameworks}

Recall that a \emph{$d$-dimensional framework} $(G,p)$ in $\mathbb{F}^d$ is the combination of a finite, simple graph $G=(V,E)$ and a realisation $p:V\rightarrow \mathbb{F}^d$.
The linear space of all such realisations of $G$ is denoted by $(\mathbb{F}^d)^V$.
We say that two realisations $p,q$ of $G$ in $\mathbb{F}^d$ are \emph{equivalent} if $||p(v) - p(w)||^2 = ||q(v) - q(w)||^2$ holds for every edge $vw \in E$.
Furthermore, we say that two realisations $p,q$ of $G$ in $\mathbb{F}^d$ are \emph{congruent} (denoted $p \sim q$) if there exists an isometry of $\mathbb{F}^d$ that sends $p$ to $q$.
For the case where $\mathbb{F}=\mathbb{R}$, this is equivalent to the condition that $||p(v) - p(w)|| = ||q(v) - q(w)||$ holds for all $v, w \in V$.

A \emph{continuous flex} of a framework $(G,p)$ in $\mathbb{F}^d$ is a continuous path $\rho :(-\varepsilon,\varepsilon) \rightarrow (\mathbb{F}^d)^V$ where $\rho_0=p$ and each framework $(G,\rho_t)$ is equivalent to $(G,p)$. We say a continuous flex $\rho$ is \emph{trivial} if $\rho_t \sim p$ for all $t \in (-\varepsilon,\varepsilon)$.
We now say $(G,p)$ is \emph{rigid} if every continuous flex of $(G,p)$ is trivial.

We will be interested in the rigidity properties that hold for `almost all realisations' of $G$.
A realisation $p$ is said to be \emph{generic} if the set of coordinates of $p$ forms an algebraically independent set over $\mathbb{Q}$ of $d|V|$ elements.
Any realisation that is congruent to a generic realisation is said to be \emph{quasi-generic}.

Determining whether a real framework is rigid is NP-hard \cite{Abbot}.
However, when a framework $(G,p)$ is generic, we can reduce this problem to linear algebra by performing a standard linearisation technique on the length constraints.
More precisely, we define the \emph{rigidity matrix} $R(G,p)$ of a $d$-dimensional framework $(G,p)$ to be the $|E|\times d|V|$ matrix whose rows are indexed by the edges and $d$-tuples of columns indexed by the vertices. The row for an edge $e=uv$ is given by:
$$
\begin{bmatrix} 
   ~ 0 & \cdots & 0 & p(u)-p(v) & 0 & \cdots & 0 & p(v)-p(u) & 0 & \cdots & 0 ~
\end{bmatrix}
$$ 
where $p(u)-p(v)$ occurs over the $d$-tuple of columns indexed by $u$ and $p(v)-p(u)$ occurs in the $d$-tuple of columns indexed by $v$.
Maxwell \cite{Maxwell} observed that for real frameworks,
$\mbox{rank } R(G,p)\leq d|V|-\binom{d+1}{2}$ whenever $p$ affinely spans $\mathbb{R}^d$;
moreover, this observation can be extended to complex frameworks.
We say that $(G,p)$ is \emph{infinitesimally rigid} if $\mbox{rank } R(G,p) = d|V|-\binom{d+1}{2}$, or $(G,p)$ is a simplex (i.e., $G=K_n$ for $n \leq d+1$) with affinely independent vertices.

A fundamental result of Asimow and Roth \cite{AsimowRothI} tells us that, when $p$ is generic and $G$ has at least $d+1$ vertices, $(G,p)$ is rigid in $\mathbb{R}^d$ if and only if it is infinitesimally rigid;
additionally, this result can be easily extended to complex frameworks.
The set of realisations in $(\mathbb{F}^d)^V$ which are not infinitesimally rigid frameworks is an algebraic set defined by polynomials with rational coefficients, and hence both rigidity and infinitesimal rigidity are generic properties for frameworks in $\mathbb{F}^d$.
We observe here that a framework in $\mathbb{R}^d$ is infinitesimally rigid when considered as a framework in $\mathbb{R}^d$ if and only if it is infinitesimally rigid when considered as a framework in $\mathbb{C}^d$.
Hence, given generic frameworks $(G,p)$, $(G,q)$ in $\mathbb{R}^d$ and $\mathbb{C}^d$ respectively, we have that $(G,p)$ is (infinitesimally) rigid if and only if $(G,q)$ is (infinitesimally) rigid.
This motivates the following definition:

\begin{definition}
    We say that a graph $G=(V,E)$ is \emph{$d$-rigid} if there exists a generic framework $(G,p)$ in either $\mathbb{R}^d$ or $\mathbb{C}^d$ that is rigid.
    If, in addition, $G-e$ is not $d$-rigid for any edge $e \in E$, then we say that $G$ is \emph{minimally $d$-rigid}.
\end{definition}


The last bit of terminology we require from classical rigidity theory is the following.
A framework $(G,p)$ in $\FF^d$ is \emph{globally rigid} if every equivalent framework to $(G,p)$ is congruent to $(G,p)$.
It was proven by Connelly \cite{conggr} and Gortler, Healy and Thurston \cite{gortler2010characterizing} that a single generic framework $(G,p)$ in $\RR^d$ is globally rigid if and only if every generic framework $(G,q)$ in $\RR^d$ is globally rigid.
Gortler and Thurston \cite{Gortler2014} later extended this to frameworks in $\CC^d$, and further proved that a generic globally rigid framework in $\RR^d$ remains globally rigid when considered as a framework in $\CC^d$.
Since every generic framework in $\RR^d$ is also a generic framework in $\CC^d$, the following definition is consistent:
a graph $G$ is \emph{globally $d$-rigid} if and only if every (equivalently, some) generic framework $(G,p)$ in $\RR^d$ (equivalently, $\CC^d$) is globally rigid.

\subsection{Dominant maps}

Throughout, we use the convention that an \emph{algebraic set} is the zero set of polynomials, and an \emph{algebraic variety} is an irreducible algebraic set (i.e. cannot be written as the union of finitely many non-trivial algebraic sets). 
Given a morphism $f:X \rightarrow Y$, we write $df(p)$ for its derivative at a point $p \in X$, i.e. the linear map given by the Jacobian of $f$ from the tangent space of $X$ at $p$ to the tangent space of $Y$ at $f(p)$.

Recall that the \emph{Zariski topology} on $\CC^n$ is the topology whose closed sets are the algebraic sets.
In particular, every non-empty open set is dense in $\CC^n$.
A morphism $f:X \rightarrow Y$ between algebraic sets $X \subseteq \CC^m$ and $Y \subseteq \CC^n$ is \emph{dominant} if $f(X)$ is Zariski dense in $Y$.
We say $f$ is \emph{generically finite} if it is dominant and there exists a Zariski open subset $U \subset Y$ where $f^{-1}(q)$ is finite for each $q \in U$.

We will make repeated use of the following standard result on dominant morphisms.
This is also known as (the algebraic version of) Sard's Theorem.

\begin{lemma}[{\cite[Theorem 17.3]{Borel1991}},{{\cite[Section 8, Corollary 1]{mumford}}}]\label{lem:dominant}
Let $X \subseteq \CC^m$ be an algebraic set and $Y \subseteq \CC^n$ be a variety.
Then the following are equivalent for any morphism $f \colon X \rightarrow Y$:
\begin{enumerate}
\item $f$ is dominant,
\item For some irreducible component $X'$ of $X$, there exists a point $p \in X'$ such that $p$ is a non-singular point of $X'$ and $\rank df(p) = \dim Y$,
\item There exists a Zariski open subset $U \subset X$ where for each $p \in U$, we have $p$ is a non-singular point of $X$ and $\rank df(p) = \dim Y$.
\end{enumerate}
Moreover, if $X$ is irreducible then there exists a non-empty Zariski open subset $V \subset f(X) \subset Y$ such that 
for every $q \in V$, every irreducible component of the algebraic set $f^{-1}(q)$ has dimension $\dim X - \dim Y$.
\end{lemma}



    Given a property $P$, if there exists a Zariski open subset $U \subset \mathbb{C}^n$ of points where property $P$ holds, we call any point in $U$ a \emph{general point (with respect to $P$)}.
    This is an extremely natural notion of genericity in algebraic geometry, and will be useful for us at several points.
    However, it is subtly different to the notion of a generic framework from rigidity theory.
    As such, we reserve the term \emph{generic point} for a point $p \in (\mathbb{C}^d)^V$ whose coordinates are algebraically independent over $\QQ$.
    In particular, $p \in (\mathbb{C}^d)^V$ is a generic point if and only if $(G,p)$ is a generic framework.

\subsection{Realisation numbers}

Given a graph $G = (V,E)$ that is $d$-rigid, we wish to study the number of equivalent realisations for a generic real framework $(G,p)$ up to isometry.
Formally, we define
\[
r_d(G,p) = \left|\left\{ q \in (\mathbb{R}^d)^V \; : \; p,q \text{ equivalent } \right\}/ \sim \right| \, ,
\]
where $p \sim q$ if $p, q$ congruent.
As showcased in \Cref{fig:notgeneric}, this is no longer a generic property: $r_d(G,p)$ varies as we range over generic realisations $p \colon V\rightarrow \RR^d$.
As such, authors define the \emph{real $d$-realisation number} to be
\[
r_d(G) = \max \left\{r_d(G,p) \; : \; p\colon V \rightarrow \RR^d \text{ generic } \right\}\, .
\]
We make a point here that the combined results of Connelly \cite{conggr} and Gortler, Healy and Thurston \cite{gortler2010characterizing} imply that $r_d(G) = 1$ if and only if $r_d(G,p)=1$ for some generic $d$-dimensional framework $(G,p)$.

The value $r_d(G)$ is somewhat difficult to work with, as often there will exist generic realisations which do not have this number of equivalent frameworks.
To avoid this issue, we often instead approach this problem algebraically by counting complex realisations also.
To do so, we introduce the \emph{(complex) rigidity map} associated to $G$: 
\begin{align*}
	f_{G,d} : (\mathbb{C}^{d})^V \rightarrow \mathbb{C}^E, \quad p \mapsto \left( \frac{1}{2}\|p (v)-p(w)\|^2 \right)_{vw \in E}.
\end{align*}
We denote the Zariski closure of the image of $f_{G,d}$ by $\ell_d(G)$.
Since the domain of $f_{G,d}$ is irreducible, $\ell_d(G)$ is a variety.
With this, we observe that two $d$-dimensional frameworks $(G,p),(G,q)$ are equivalent if and only if $f_{G,d}(p)=f_{G,d}(q)$.
If the set of vertices of $(G,p)$ affinely span $\mathbb{C}^d$,
then two realisations $p,q$ are congruent if and only if $f_{K_V,d}(p) = f_{K_V,d}(q)$, where $K_V$ is the complete graph with vertex set $V$ (see \cite[Section 10]{Gortler2014} for more details).
We observe here that the Jacobian of $f_{G,d}$ at $p$ 
is exactly the rigidity matrix $R(G,p)$.

For all $p \in (\mathbb{C}^{d})^V$,
we define 
\begin{equation*}
	C_d(G,p) := f^{-1}_{G,d} (f_{G,d}(p))/\sim
\end{equation*}
to be the \emph{realisation space of $(G,p)$}.
The cardinality of $C_d(G,p)$ is constant over a Zariski open dense subset of $(\mathbb{C}^d)^V$ (e.g., \cite[Proposition 3.5]{dewar23}).
This leads to the first definition of the $d$-realisation number:

\begin{definition}\label{def:count}
	The \emph{(complex) $d$-realisation number} of a $d$-rigid graph $G=(V,E)$ is an element of $\mathbb{N} \cup \{\infty\}$ given by
	\begin{align*}
		c_d(G) :=		
		\begin{cases}
			|C_d(G,p)| \text{ for a general point $p \in (\mathbb{C}^d)^V$} &\text{if } |V| \geq d+1, \\
			 1 &\text{if } |V| \leq d \text{ and $G$ is complete}, \\
			 \infty &\text{if } |V| \leq d \text{ and $G$ is not complete}.
		\end{cases}
	\end{align*}
\end{definition}

While this definition is natural from an algebraic geometry perspective, it is not clear a priori that all generic realisations are contained in the Zariski open set on which $|C_d(G,p)|$ is constant.
    We would prefer that the definition relied only on the notion of a generic framework and does not require specifying the Zariski open set on which the realisation space has constant cardinality.
    The following theorem states that we can replace `general' with `generic realisation' in the definition of $c_d(G)$.

\begin{restatable}{theorem}{genericimpliesgeneral}\label{thm:generic-implies-general}
Let $G=(V,E)$ be a $d$-rigid graph.
Then $c_d(G) = |C_d(G,p)|$ for all quasi-generic realisations $p \in (\CC^d)^V$ of $G$.
\end{restatable}

The case when $d=2$ was proved by Jackson and Owen \cite[Theorem 3.6]{JO19}.
Their techniques may be adapted to arbitrary dimension so we defer the proof to \Cref{sec:appendix}.

The main upshot of \Cref{thm:generic-implies-general} is that we now have a `general point' definition and a `generic framework' definition of $c_d(G)$ that coincide.
We will use both at different points throughout our analysis, as each one lends itself better to different toolkits.

In practice, $C_d(G,p)$ is often awkward to work with as quotienting by isometries obscures the algebraic variety structure.
We will instead restrict the domain of $f_{G,d}$ to `pinned frameworks', removing almost all isometries beforehand.
We do this as follows.
Let $G=(V,E)$ be a graph with at least $d+1$ vertices and fix a sequence of $d$ vertices $v_1,\ldots, v_d$.
We now define the linear space
\begin{align}\label{eq:xset}
	X_{G,d} := \left\{ p \in (\mathbb{C}^{d})^V : p_j(v_k) =0 \text{ for all } 1\leq k \leq j \leq d \right\}\,.
\end{align}
Intuitively, this is the space of frameworks where $v_1$ is pinned to the origin, $v_2$ is pinned to the $x$-axis, $v_3$ is pinned to the $xy$-plane, etc.
We further note that $X_{G,d}$ has dimension $d|V| - \binom{d+1}{2}$.

Almost every framework is congruent to a framework in $X_{G,d}$; in particular, every quasi-generic framework is congruent to a framework in $X_{G,d}$. (The precise condition is given in \Cref{lem:canonical+position}.)
We note that the frameworks in $X_{G,d}$ are never generic as they have been pinned in special position, but $X_{G,d}$ does contain a dense subset of quasi-generic frameworks.

With this, we define the \emph{pinned rigidity map}
\begin{align*}
	\tilde{f}_{G,d} : X_{G,d} \rightarrow \ell_d(G), \quad p \mapsto f_{G,d}(p),
\end{align*}
i.\,e., the restriction of $f_{G,d}$ to the domain $X_{G,d}$ and the codomain $\ell_d(G)$.
As can be seen by the following lemma, the map $\tilde{f}_{G,d}$ allows us to more easily define the cardinality of the set $C_d(G,p)$ for most realisations $p$.

\begin{lemma}[{\cite[Lemma 3.2]{dewar23}}]\label{l:xgd}
	Let $G=(V,E)$ be a graph with $|V| \geq d+1$. 
	The image of $\tilde{f}_{G,d}$ is Zariski dense in the image of $f_{G,d}$, and so $\tilde{f}_{G,d}$ is dominant.
	Moreover, there exists a Zariski open subset $U \subset (\mathbb{C}^d)^V$ 
    such that
	\begin{align*}
		\left|\tilde{f}_{G,d}^{-1}\left(f_{G,d}(p) \right) \right| = 2^{d} |C_d(G,p)|
	\end{align*}
	for all $p \in U$. 
\end{lemma}

The derivative of the pinned rigidity map can be related to infinitesimal rigidity if our pinned vertices are affinely independent:

\begin{lemma}\label{lem:pinnedrigmatrix}
Choose $p \in X_{G,d}$ such that the points $p(v_1), \dots, p(v_d)$ are affinely independent.
Then $d\tilde{f}_{G,d}(p)$ has rank $d|V| - \binom{d+1}{2}$ if and only if $(G,p)$ is infinitesimally rigid.
\end{lemma}

\begin{proof}
    Observe that the matrix representing $d\tilde{f}_{G,d}(p)$ can be obtained from $R(G,p)$ by deleting the columns labelled by $(j,v_k)$ for all $j \geq k$.
    If $d\tilde{f}_{G,d}(p)$ has rank $d|V| - \binom{d + 1}{2}$, it necessarily implies that $R(G,p)$ is also full rank and so $(G,p)$ is infinitesimally rigid.

    Now suppose $(G,p)$ is infinitesimally rigid.
    Choose any $d \times d$ complex matrix $A$ which is skew-symmetric,
    i.e., $A^{\top}=-A$.
    Then we observe that for any $x \in \mathbb{C}^d$ we have
    \begin{equation*}
        x^{\top} A x = (x^{\top} A x)^{\top} = x^{\top} A^{\top} x = - x^{\top} A x,
    \end{equation*}
    and so $x^{\top} A x = 0$.
    This implies that the set
    \begin{equation*}
        \mathcal{T}(p) := \left\{ (A p(v) + b)_{v \in V} : A \text{ is $d \times d$ skew-symmetric and } b \in \mathbb{C}^d \right\}
    \end{equation*}
    is contained in $\ker R(G,p)$.

    \begin{claim}
        $\ker R(G,p)= \mathcal{T}(p)$.
    \end{claim}
    
    \begin{claimproof}
        As $(G,p)$ is infinitesimally rigid, it suffices to show $\dim \mathcal{T}(p)= \binom{d+1}{2}$.
        In fact, since $p(v_1) = \mathbf{0}$, we only need to show that the restriction of $\mathcal{T}(p)$ to vectors where $b = \mathbf{0}$ has dimension $\binom{d}{2}$.
        
        Suppose for contradiction that this is not the case.
        Then there must exist some non-zero skew-symmetric matrix $A$ where $A p(v_i)= \mathbf{0}$ for each $i \in [d]$.
        Since points $p(v_2), \dots, p(v_d)$ are a linear basis of $\mathbb{C}^{d-1} \times \{0\}$,
        the left $d-1$ columns of $A$ only contain zero entries.
        However, $A$ being skew-symmetric then implies that $A$ is the all-zeroes matrix, a contradiction.
    \end{claimproof}
    
    Now choose any vector $u = (A p(v) + b)_{v \in V} \in \ker R(G,p) \cap X_{G,d}$.
    We immediately see that $b = \mathbf{0}$ as $p(v_1)= A p(v_1) + b= \mathbf{0}$.
    Since $(A p(v))_{v \in V}$ also satisfies the pinning conditions described by $X_{G,d}$,
    we have that for each $1 \leq i \leq j \leq d$, the $j$-th coordinate of $Ap(v_i)$ is 0.

    \begin{claim}
        $A$ is the all-zeroes matrix.
    \end{claim}

    \begin{claimproof}
        Fix $A_1,\ldots,A_d$ to be the columns of $A$.
        We now prove $A$ is the all-zeroes matrix by induction on its columns.
        
        First, take the vector $Ap(v_2)$.
        Since $\mathbf{0} = p(v_1),p(v_2)$ are affinely independent,
        we have that the first coordinate of $p(v_2)$, here denoted $\lambda$, is non-zero.
        With this, we have that $Ap(v_2) = \lambda A_1$, and thus $A_1 = \lambda^{-1} Ap(v_2)$.
        As the diagonal of $A$ only takes zero values,
        the first coordinate of $A_1$ is 0.
        As $j$-th coordinate of $Ap(v_2)$ is 0 for each $2 \leq j \leq d$,
        we have that all other coordinates of $A_1$ are 0.
        Thus $A_1 = \mathbf{0}$.
        
        Now suppose that $A_1 = \cdots = A_{i-1} = \mathbf{0}$.
        Since the vectors $p(v_1), \ldots,p(v_{i+1})$ are affinely independent with $p(v_1), \ldots,p(v_i) \in \mathbb{C}^{i-1} \times \{0\}^{d-i+1}$,
        we have that the $i$-th coordinate of $p(v_{i+1})$, here denoted $\psi$, is non-zero.
        By our assumption, we now have that $Ap(v_{i+1}) = \psi A_i$, and thus $A_i = \psi^{-1} Ap(v_{i+1})$.
        Since the first $i-1$ rows of $A$ have zero entries,
        the first $i-1$ coordinates of $A_i$ are 0.
        As the diagonal of $A$ only takes zero values,
        the $i$-th coordinate of $A_i$ is 0.
        As $j$-th coordinate of $Ap(v_i)$ is 0 for each $i+1 \leq j \leq d$,
        we have that all other coordinates of $A_i$ are 0.
        Thus $A_i = \mathbf{0}$.
        The claim now follows by induction.
    \end{claimproof}
    
    The result now follows as the kernel of $d\tilde{f}_{G,d}(p)$ is exactly the linear space $\ker df_{G,d}(p) \cap X_{G,d}$.
\end{proof}

Finally, we note a useful (but non-trivial) fact that any realisation equivalent to a generic realisation in $\CC^d$ is also quasi-generic.
This was shown in \cite{JO19} in the case where $d=2$;
the general case proved similarly, but we include a proof for completeness in \Cref{sec:appendix}.

\begin{restatable}{lemma}{genericimpliesquasi}\label{lem:rigid+generic+implies+quasi}
    Let $G$ be a $d$-rigid graph with at least $d+1$ vertices, let $(G,p)$ be generic in $\mathbb{C}^d$ and let $(G,q)$ be equivalent to $(G,p)$.
    Then $(G,q)$ is quasi-generic.
\end{restatable}

\section{Dividing the realisation number}\label{sec:divide}

In this section we obtain two different division rules for rigid subgraphs of rigid graphs.

\subsection{Spanning rigid subgraphs}

We begin this section by proving that the realisation number of any spanning $d$-rigid subgraph is divisible by the realisation number of the larger $d$-rigid graph. 

\gcd*

\begin{proof}
    Choose a generic realisation $p \in (\CC^{d})^V$ for $G$ with $d$-rigid subgraph $H$.
    Let $[p^1], \ldots, [p^k]$ denote the equivalence classes of $C_d(H,p)$. Then $c_d(H)=k$.
    Define an equivalence relation $\sim_{G}$ on $f^{-1}_{H,d}(f_{H,d}(p))$ by putting $p' \sim_G p''$ if and only if $f_{G,d}(p') = f_{G,d}(p'')$.
    Now note that if $p',p'' \in [p^i]$,
    then $p' \sim_G p''$.
    Hence, $\sim_G$ gives an equivalence relation on the elements of $C_d(H,p)$.
    Let $[[p^{k_1}]], \ldots, [[p^{k_\ell}]]$ be the equivalence classes of $\sim_G$ on $C_d(H,p)$.
    Pick any $s \in [\ell]$,
    and label the elements of $[[p^{k_s}]]$ by $[q^1], \ldots, [q^t]$.
    If we choose some $q \in [q^j]$ for any $j \in [t]$,
    then $C_d(G,q) = \{ [q^1], \ldots, [q^t] \}$.
    Since $f_{H,d}(p) = f_{H,d}(q)$ and $p$ is generic,
    the realisation $q$ is congruent to a generic realisation by \Cref{lem:rigid+generic+implies+quasi}.
    Hence,
    \begin{equation*}
        c_d(G) = |C_d(G,p)| = |C_d(G,q)| = t.
    \end{equation*}
    Since this holds for each of the equivalence classes $[[p^{k_1}]], \ldots, [[p^{k_\ell}]]$,
    it follows that $c_d(H) = k = \ell t = \ell c_d(G)$.
\end{proof}

It is easy to see that some realisation numbers are not the common divisor of their spanning minimally $d$-rigid subgraphs:
the graph $K_4$ is globally 2-rigid (and hence $c_2(G) = 1$), but the only minimally 2-rigid spanning subgraph of $G$ is $K_4-e$ which has 2-realisation number 2.

\begin{example}
Consider the graphs $G_1,G_2,G_3$ in \Cref{fig:prismadapted}. The graph $G_1$ is globally rigid in the plane (see, for example, \cite[Theorem~6.1]{bj03}). We show this fact follows from \Cref{thm:gcd}.
Note that $G_2$ and $G_3$ are obtained from $G_1$ by deleting single edges.
Their 2-realisation numbers are $c_2(G_2)=45$ and $c_2(G_3)=32$: the former is computed in \cite[Section 8]{JO19} and both can be found in \cite{capcodata}.
Hence the greatest common divisor is $1$ and \Cref{thm:gcd} implies that $c_2(G_1)=1$ i.e. $G_1$ is globally rigid.
\end{example}

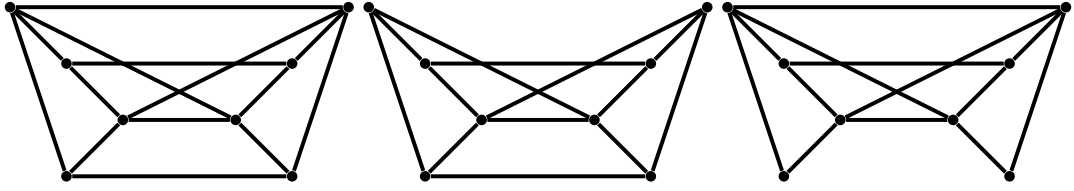
\begin{figure}[tp]
	\begin{center}
        \begin{tikzpicture}[scale=0.75]
			\node[vertex] (11) at (-1,0) {};
			\node[vertex] (21) at (-2,1) {};
			\node[vertex] (31) at (-2,-1) {};
			
			\node[vertex] (12) at (1,0) {};
			\node[vertex] (22) at (2,1) {};
			\node[vertex] (32) at (2,-1) {};

			\node[vertex] (a) at (-3,2) {};
                \node[vertex] (b) at (3,2) {};
			
			\draw[edge] (11)edge(12);
			\draw[edge] (21)edge(22);
			\draw[edge] (31)edge(32);
			
			\draw[edge] (11)edge(21);
			\draw[edge] (31)edge(11);
			
			\draw[edge] (12)edge(22);
			\draw[edge] (32)edge(12);

			\draw[edge] (a)edge(21);
			\draw[edge] (a)edge(12);
			\draw[edge] (a)edge(31);

                \draw[edge] (b)edge(22);
			\draw[edge] (b)edge(11);
			\draw[edge] (b)edge(32);

                \draw[edge] (a)edge(b);
		\end{tikzpicture}
            \begin{tikzpicture}[scale=0.75]
			\node[vertex] (11) at (-1,0) {};
			\node[vertex] (21) at (-2,1) {};
			\node[vertex] (31) at (-2,-1) {};
			
			\node[vertex] (12) at (1,0) {};
			\node[vertex] (22) at (2,1) {};
			\node[vertex] (32) at (2,-1) {};

			\node[vertex] (a) at (-3,2) {};
                \node[vertex] (b) at (3,2) {};
			
			\draw[edge] (11)edge(12);
			\draw[edge] (21)edge(22);
			\draw[edge] (31)edge(32);
			
			\draw[edge] (11)edge(21);
			\draw[edge] (31)edge(11);
			
			\draw[edge] (12)edge(22);
			\draw[edge] (32)edge(12);

			\draw[edge] (a)edge(21);
			\draw[edge] (a)edge(12);
			\draw[edge] (a)edge(31);

                \draw[edge] (b)edge(22);
			\draw[edge] (b)edge(11);
			\draw[edge] (b)edge(32);

		\end{tikzpicture}
            \begin{tikzpicture}[scale=0.75]
			\node[vertex] (11) at (-1,0) {};
			\node[vertex] (21) at (-2,1) {};
			\node[vertex] (31) at (-2,-1) {};
			
			\node[vertex] (12) at (1,0) {};
			\node[vertex] (22) at (2,1) {};
			\node[vertex] (32) at (2,-1) {};

			\node[vertex] (a) at (-3,2) {};
                \node[vertex] (b) at (3,2) {};
			
			\draw[edge] (11)edge(12);
			\draw[edge] (21)edge(22);
			
			\draw[edge] (11)edge(21);
			\draw[edge] (31)edge(11);
			
			\draw[edge] (12)edge(22);
			\draw[edge] (32)edge(12);

			\draw[edge] (a)edge(21);
			\draw[edge] (a)edge(12);
			\draw[edge] (a)edge(31);

                \draw[edge] (b)edge(22);
			\draw[edge] (b)edge(11);
			\draw[edge] (b)edge(32);

                \draw[edge] (a)edge(b);
		\end{tikzpicture}
	\end{center}
	\caption{Three 2-rigid graphs: (left) $G_1$ with $c_2(G_1) = 1$; (middle) $G_2$ with $c_2(G_2) = 45$; (right) $G_3$ with $c_2(G_3) = 32$.}\label{fig:prismadapted}
\end{figure}

\begin{corollary}\label{cor:globallylinked}
    Let $G$ be a $d$-rigid graph.
    Suppose there exists an edge $ij \in E$ and realisations $p,q \in (\CC^{d})^V$ such that $p$ is generic, $f_{G,d}(p) = f_{G,d}(q)$ but $\|p_i - p_j\|^2 \neq \|q_i - q_j\|^2$.
    Then $c_d(G+ij) \leq \frac{1}{2}c_d(G)$,
    and this bound is tight.
\end{corollary}

\begin{proof}
    The hypotheses on $p$ and $q$ imply that $c_d(G+ij) < c_d(G)$.
    Hence, $c_d(G+ij) \leq \frac{1}{2}c_d(G)$ by \Cref{thm:gcd}.
    To see that the result is tight,
    choose $G$ such that $c_d(G) = 2$ (for example,
    the union of two complete graphs with more than $d$ vertices that share exactly $d$ vertices).
\end{proof}

Repeated uses of \Cref{cor:globallylinked} leads to the following result.

\begin{corollary}\label{cor:lowerbound}
    Let $G$ be a $d$-rigid graph with a sequence of non-edges $\{v_1 w_1\},\ldots,\{v_k w_k\}$.
    Let $G_0, \ldots , G_k$ be a sequence of graphs where $G_0 = G$ and $G_{i} = G_{i-1} + v_{i}w_{i}$ for each $i \in [k]$.
    If $G_k$ is globally $d$-rigid and $c_d(G_{i}) < c_d(G_{i-1})$ for each $i \in [k]$,
    then $c_d(G) \geq 2^k$.
\end{corollary}

While it is unlikely to lead to an efficient algorithm, we use these ideas to give a simple upper bound on the number of edges one must add to a $d$-rigid graph $G$ so that the resulting supergraph is globally $d$-rigid. We note that this global rigidity augmentation problem is well studied but seemingly only in low dimensions. In particular,
\cite{km22} gives a polynomial time algorithm for the global rigidity augmentation problem in 2-dimensions. 
The 1-dimensional case is better known as the 2-connectivity augmentation problem \cite{et76}.

\begin{corollary}\label{cor:augment}
    Let $G$ be a $d$-rigid graph.
    Let $F$ be the smallest set of non-edges of $G$ such that $G+F$ is globally $d$-rigid.
    If $c_d(G) = p_1 \cdots p_k$ for (possibly not distinct) primes $p_i$,
    then $|F| \leq k$. 
\end{corollary}

\begin{proof}
    As $c_d(G) > 1$,
    there exists a pair of non-adjacent vertices $i,j$ such that $c_d(G + ij) < c_d(G)$.
    By \Cref{thm:gcd},
    it follows that the prime decomposition of $c_d(G + ij)$ has length at most $k-1$.
    The proof now follows from repeated application.
\end{proof}

\Cref{cor:augment} is not tight in general.
For example, the graph $G_2$ given in \Cref{fig:prismadapted} has 2-realisation number $45 = 3 \cdot 3 \cdot 5$,
however, adding any non-edge to $G_2$ always produces a globally 2-rigid graph. 
On the other hand, there do exist some graphs for which it is tight. Let $G$ be the complete multipartite graph $K_{1,\dots,1,t}$ with $d$ parts of size 1 and one part of size $t\geq 1$. Then $c_d(G)=2^{t-1}$. Any supergraph obtained from $G$ by adding strictly less than $t-1$ edges is not $(d+1)$-connected and hence not globally $d$-rigid. Hence \Cref{cor:augment} implies that the smallest set of non-edges required to augment $G$ to a globally $d$-rigid graph has cardinality $t-1$.

\subsection{Proper rigid graphs}

We now prove \Cref{thm:isosubreal}.

\isosubreal*

\begin{proof}
    Choose vertices $v_1,\ldots,v_d \in V(H)$ to define both $X_{G,d}$ and $X_{H,d}$, and let $p$ be a general point in $X_{G,d}$.

    \begin{claim}
        The map
        \begin{equation*}
            h : \tilde{f}^{-1}_{G-E(H),d}(\tilde{f}_{G-E(H),d}(p)) \rightarrow X_{H,d}, \quad q \mapsto (q(v))_{v \in V(H)}
        \end{equation*}
        is generically finite.
    \end{claim}

    \begin{claimproof}
        For notational convenience, we write 
        \begin{align*}
        Z_p &= f^{-1}_{G-E(H),d}(f_{G-E(H),d}(p)) \subseteq (\CC^d)^{V} \, ,\\ \widetilde{Z}_p &= \tilde{f}^{-1}_{G-E(H),d}(\tilde{f}_{G-E(H),d}(p)) \subseteq X_{G,d} \cong \CC^{d|V| - \binom{d+1}{2}} \, ,
        \end{align*}
        hence $\widetilde{Z}_p$ is the domain of $h$.
        We first note that $\dim(\widetilde{Z}_p) = |E(\widetilde{H})| = \dim(X_{H,d})$.
        The second equality follows from $\widetilde{H}$ being minimally $d$-rigid.
        For the first equality, $\widetilde{Z}_p$ has the dimension of a generic fiber of $\tilde{f}_{G-E(H),d}$. 
        As $\tilde{f}_{G-E(H),d}$ is dominant by \Cref{l:xgd}, it follows from \added{\Cref{lem:dominant}} that
        \[
        \dim \widetilde{Z}_p = \dim X_{G-E(H),d} - \dim \ell_{d}(G-E(H)) = d|V| - \binom{d+1}{2} - |E(G)| + |E(H)| = |E(\widetilde{H})| \, ,
        \]
        where the final equality follows from $\widetilde{G} := G - E(H) + E(\widetilde{H})$ being minimally $d$-rigid.
        As the domain and codomain of $h$ have the same dimension, it suffices to prove that $h$ is dominant by \added{\Cref{lem:dominant}}.
        We do this by showing $p$ is non-singular and $\rank dh(p) = |E(\widetilde{H})|$.
        
        As the Jacobian of $f_{G-E(H),d}$ at $p$ is the rigidity matrix $R(G-E(H),p)$, it follows that the tangent space of $Z_p$ at $p$ is $\ker R(G-E(H),p)$.
        As $\widetilde{Z}_p = Z_p \cap X_{G,d}$, it follows that the tangent space of $\widetilde{Z}_p$ at $p$ is $\ker R(G-E(H),p) \cap X_{G,d}$.
        Hence, the derivative of $h$ at $p$ is the linear projection map
        \begin{equation*}
            dh(p) : \ker R(G-E(H),p) \cap X_{G,d} \rightarrow X_{H,d}, ~ q \mapsto (q(v))_{v \in V(H)}.
        \end{equation*}
        We observe that the space $\ker R(G-E(H),p)\cap X_{G,d}$ has dimension $|E(\widetilde{H})|$ as follows.
        As $\widetilde{G}$ is a minimally $d$-rigid, all the rows of $R(\widetilde{G},p)$ are linearly independent, and the intersection of its kernel with $X_{G,d}$ is zero-dimensional.
        We get $R(G-E(H),p)$ from $R(\widetilde{G},p)$ by removing $|E(\widetilde{H})|$ rows, hence $\ker R(G-E(H),p)\cap X_{G,d}$ has dimension $|E(\widetilde{H})|$.
        This implies that $\widetilde{Z}_p$ and the tangent space at $p$ have the same dimension, hence $p$ is non-singular.

        It remains to show that $dh(p)$ has rank $|E(\widetilde{H})|$.
        Recall that as $\widetilde{G}$ is minimally $d$-rigid, we have $\ker(R(\widetilde{G},p)) \cap X_{G,d} = 0$ and each row of $R(\widetilde{G},p)$ is linearly independent.
        Hence for each edge $xy \in E(\widetilde{H})$, there exists a unique (up to scaling) infinitesimal flex $u_{xy} \in \ker R(\widetilde{G}-xy,p) \cap X_{G,d}$.
        Explicitly, we have
        \begin{align*}
            \big(p(x)-p(y) \big) \cdot \big( u_{xy}(x) - u_{xy}(y) \big) &\neq 0, \\
            \big(p(v)-p(w) \big) \cdot \big( u_{xy}(v) - u_{xy}(w) \big) &= 0 \qquad \text{ for all } vw \in E(G) \setminus \{xy\}.
        \end{align*}
        These conditions imply that $\left\{dh(p) \big( u_{xy} \big): xy \in E(\widetilde{H}) \right\}$ is a linearly independent set.
        This then implies that the rank of $dh(p)$ is $|E(\widetilde{H})|$ as required.
    \end{claimproof}

    Now fix a general point $q \in \tilde{f}^{-1}_{G-E(H),d}(\tilde{f}_{G-E(H),d}(p))$.
    It is immediate that
    \begin{equation*}
        \tilde{f}_{G,d}^{-1}(\tilde{f}_{G,d}(q)) = (\tilde{f}_{H,d} \circ h)^{-1} (\tilde{f}_{H,d} \circ h) (q).
    \end{equation*}
    Since both $h$ and $\tilde{f}_{H,d}$ are generically finite (the latter being a consequence of \Cref{l:xgd} and $H$ being $d$-rigid),
    their composition multiplies their degrees:
    \begin{equation*}
         \left|(\tilde{f}_{H,d} \circ h)^{-1} (\tilde{f}_{H,d} \circ h)(q) \right| = \deg (\tilde{f}_{H,d} \circ h) = (\deg \tilde{f}_{H,d}) (\deg h).
    \end{equation*}
    Hence, by \Cref{l:xgd} we have
    \begin{equation}\label{eq:subeq}
        c_d(G) = 2^d \left| \tilde{f}_{G,d}^{-1}(\tilde{f}_{G,d}(q)) \right| = 2^d \deg (\tilde{f}_{H,d} \circ h) = 2^d (\deg \tilde{f}_{H,d}) (\deg h) = c_d(H) (\deg h).
    \end{equation}
    This now concludes the proof.
\end{proof}

\section{Graph operations}
\label{sec:operations}

In this section we investigate the behaviour of realisation numbers with respect to a variety of different graph operations.

\subsection{0-extensions}

The simplest graph operation that preserves rigidity is the \emph{$d$-dimensional 0-extension}.
This operation on $G$ adds a new vertex $x$ and $d$ new edges $xw_1, \dots, xw_d$ to distinct vertices $w_i$ of $G$.
A schematic of a 3-dimensional 0-extension is given in \Cref{fig:0-extension}.
It is well known that this operation preserves (minimal) rigidity \cite[Lemma 11.1.1]{Wlong}.
Moreover, it precisely doubles the realisation number.

\begin{figure}
    \centering
    \begin{tikzpicture}[scale=1]

\draw (0,0) circle (30pt);
\draw (4,0) circle (30pt);

\filldraw (4,2) circle (2pt) node[anchor=south]{$x$};

\filldraw (.5,.5) circle (2pt) node[anchor=north]{$w_3$};
\filldraw (4.5,.5) circle (2pt) node[anchor=north]{$w_3$};

\filldraw (-.5,.5) circle (2pt) node[anchor=north]{$w_1$};
\filldraw (3.5,.5) circle (2pt) node[anchor=north]{$w_1$};

\filldraw (0,0) circle (2pt) node[anchor=north]{$w_2$};
\filldraw (4,0) circle (2pt) node[anchor=north]{$w_2$};



\draw[black,thick]
(4.5,.5) -- (4,2); 

\draw[black,thick]
(4,2) -- (4,0);

\draw[black,thick]
(3.5,.5) -- (4,2); 

\draw[black]
(1.5,0) -- (2.5,0);

\draw[black,thick]
(2.4,-.2) -- (2.5,0) -- (2.4,.2);
\end{tikzpicture}
    \caption{A schematic of 3-dimensional 0-extension.}
    \label{fig:0-extension}
\end{figure}
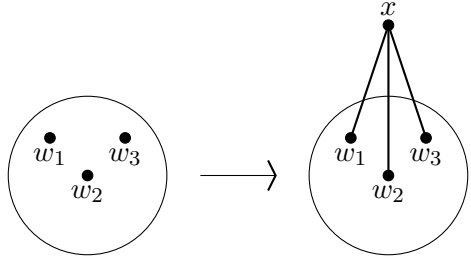

\begin{lemma}[{\cite[Lemma 7.1]{dewar23}}]\label{lem:0ext}
    Let $G = (V,E)$ be a $d$-rigid graph with at least $d+1$ vertices.
    If $G'$ is obtained from $G$ via a 0-extension, then $c_d(G') = 2c_d(G)$.
\end{lemma}

In fact, we can show that $d$-dimensional 0-extensions also double the number of real realisations.
To show this, we first need to introduce the following terminology.
A \emph{Euclidean distance matrix} is any square matrix of the form
\begin{equation*}
    D := 
    \begin{pmatrix}
        \|y_i - y_j\|^2
    \end{pmatrix}_{i,j \in [n]}
\end{equation*}
for some set of points $y_1,\ldots,y_n$ in $\mathbb{R}^d$;
we additionally say that $D$ is the Euclidean distance matrix for the points $y_1,\ldots,y_n$.
We can exactly characterise which matrices are Euclidean distance matrices using a famous result of Schoenberg.

\begin{theorem}[Schoenberg \cite{schon}]\label{thm:schoen}
    Let $D$ be an $n \times n$ matrix with non-negative entries and zero entries for its diagonal.
    Then $D$ is a Euclidean distance matrix for a set of points with $d$-dimensional affine span if and only if the $(n-1) \times (n-1)$ matrix
    \begin{equation*}
        G :=
        \begin{pmatrix}
            D_{i,n} + D_{j,n} - D_{i,j}            
        \end{pmatrix}_{i,j \in [n-1]}
    \end{equation*}
    is positive semi-definite with rank $d$.
\end{theorem}

We also require the following corollary to \Cref{thm:schoen}.

\begin{corollary}[see, for example, {\cite[Corollary 3.3.3]{so2007}}]\label{cor:coredm}
    If $D$ is an $n \times n$ Euclidean distance matrix,
    then it is negative semi-definite on the points in the linear space orthogonal to $(1,\ldots,1) \in \mathbb{R}^d$.
\end{corollary}

\begin{lemma}\label{lem:edm}
    Given affinely independent points $y_1,\ldots,y_d$ in $\mathbb{R}^d$,
    let
    \begin{equation*}
        f : \mathbb{R}^d \rightarrow \mathbb{R}^d, ~ z \mapsto \left( \|z-y_i\|^2 \right)_{i \in [d]}.
    \end{equation*}
    Then there exists an non-empty open subset $U \subset \mathbb{R}^d$ contained in the image of $f$ where the following holds:
    \begin{enumerate}
        \item For any $\lambda \in U$,
        there exist exactly two solutions to the equation $f(z) = \lambda$.
        \item For some $R>0$,
        the set $U$ contains the ray $\{(r, \ldots, r): r > R \}$.
    \end{enumerate}
\end{lemma}

\begin{proof}
    Fix $U$ to be the image under $f$ of the set of points that are not in the affine span of $y_1,\ldots,y_d$.
    We observe that the Jacobian $df(z)$ is singular if and only if $z$ lies in the affine span of $y_1,\ldots,y_d$.
    Hence, by the inverse function theorem (e.g., \cite[Theorem 9.24]{rudin}) the set $U$ is non-empty and open.
    
    Take any point $z^*$ not in the affine span of $y_1,\ldots,y_d$ and suppose that $z$ satisfies the equation $f(z) = f(z^*)$.
    Then there exists an isometry of $\mathbb{R}^d$ which maps $z$ to $z^*$ and maps each $y_i$ to itself.
    Since the points $y_1,\ldots,y_d$ are affinely independent,
    there exists a single isometry that fixes each point $y_i$: the reflection through the affine span of $y_1,\ldots,y_d$.
    Hence either $z=z^*$ or $z$ is the unique point found by reflecting through the affine span of $y_1,\ldots,y_d$.

    For each $r >0$,
    fix the matrix 
    \begin{equation*}
        G(r) = 
        \begin{pmatrix}
            2r - \|y_i - y_j\|^2
        \end{pmatrix}_{i,j \in [d]}.
    \end{equation*}
    By \Cref{thm:schoen}, $(r,\ldots,r)$ is contained in $U$ if and only if $G(r)$ is positive definite.
    It is immediate that $G(r) = 2r J - Y$,
    where $J$ is the all-ones matrix and
    \begin{equation*}
        Y := 
        \begin{pmatrix}
            \|y_i - y_j\|^2
        \end{pmatrix}_{i,j \in [d]}. 
    \end{equation*}
    Since $Y$ is the Euclidean distance matrix,
    the matrix $Y$ is negative semi-definite on the linear space orthogonal to $\mathbf{e} := (1,\ldots,1)$ by \Cref{cor:coredm}.
    It is clear that $\mathbf{e}$ is the eigenvector of $J$ with eigenvalue $d$ and $J$ is zero on the linear space orthogonal to $\mathbf{e}$.
    Hence $G(r)$ is positive definite if and only if 
    \begin{equation*}
        \mathbf{e}^{\top} G(r) \mathbf{e} = \mathbf{e}^{\top} (2rJ-Y) \mathbf{e} =  2 r d^2 - \mathbf{e}^{\top} Y \mathbf{e} > 0.
    \end{equation*}
    It is now easy to see that the above equation holds for sufficiently large $r>0$.
\end{proof}

\begin{proposition}\label{prop:0extreal}
    Let $G = (V,E)$ be a $d$-rigid graph with at least $d+1$ vertices.
    If $G'$ is obtained from $G$ via a 0-extension, then $r_d(G') = 2r_d(G)$.
\end{proposition}

\begin{proof}
    We first prove that $r_d(G') \geq 2r_d(G)$.
    Fix $x$ to be the new vertex added to $G$ to form $G'$, and let $w_1,\ldots,w_d$ be its neighbours in $G'$.
    Choose a real generic realisation $p$ of $G$ where $r_d(G,p)= r_d(G) = t$.
    Let $(G,p_1),\dots,(G,p_t)$ be pairwise non-congruent but equivalent $d$-dimensional real frameworks where $p_1=p$.
    As each $p_k$ is quasi-generic (\Cref{lem:rigid+generic+implies+quasi}),
    the points $p_k(w_1), \ldots, p_k(w_d)$ are affinely independent for each $k \in [t]$.
    
    For each $k \in [t]$,
    fix $f_k$ and $U_k$ to be the map $f$ and non-empty open set $U_k$ from \Cref{lem:edm} for the points $y_1 = p_k(w_1), \ldots, y_d = p_k(w_d)$.
    By \Cref{lem:edm},
    the set $U := \bigcap_{k=1}^t U_k$ is a non-empty open set in $\RR^d$;
    that it is non-empty stems from the fact that $U$ must contain some ray $\{(r,\ldots,r): r >R\}$ for some ray $R>0$.
    Hence we can pick $\lambda \in U$ whose coordinates are algebraically independent over the set of coordinates for $p$.
    By \Cref{lem:edm}, there exists two solutions to the equation $f_k(z) = \lambda$ that we label $a_k, b_k$.
    For each $1 \leq k \leq t$, we define the realisations $p'_{k,a}$ and $p'_{k,b}$ of $G'$ by $p'_{k,a}(v) = p'_{k,b}(v) = p_k(v)$ for all $v \in V$, and $p'_{k,a}(x) = a_k$ and $p'_{k,b}(x) = b_k$.
    Additionally, we set $p' = p'_{1,a}$: our choice of $\lambda$ as algebraically independent over $p$ implies that the coordinates of $a_1$ must also be algebraic independent over $p$, hence $p'$ is a generic realisation of $G'$.
    It follows that
    \begin{equation*}
        r_d(G') \geq r_d(G',p') \geq \left| \left\{p'_{1,a},p'_{1,b}, p'_{2,a},p'_{2,b},\ldots, p'_{t,a},p'_{t,b}  \right\}  \right| = 2t = 2r_d(G).
    \end{equation*}

    We now show that $r_d(G') \leq 2r_d(G)$.
    Choose a generic realisation $p'$ of $G'$ in $\mathbb{R}^d$ with $r_d(G',p') = r_d(G')$ and fix $p$ to be the restriction of $p'$ to $V$.
    Now choose any $(G',q')$ in $\mathbb{R}^d$ that is equivalent to $(G',p')$.
    Then, given $q$ is the restriction of $q'$ to $V$, we have that $(G,q)$ is equivalent to $(G,p)$.
    By \Cref{lem:rigid+generic+implies+quasi},
    both $(G',q')$ and $(G,q)$ are quasi-generic, and thus $q(w_1),\ldots,q(w_d)$ are affinely independent.
    Let $(G',q^*)$ be another framework where $q^*(v) = q(v)$ for each $v \in V$.
    Then there exists an isometry that maps $q'(x)$ to $q^*(x)$ and fixes each point $q(w_1),\ldots,q(w_d)$.
    Since $q(w_1),\ldots,q(w_d)$ are affinely independent,
    there exists a single isometry that fixes each point $q(w_i)$: the reflection through the affine span of $q(w_1),\ldots,q(w_d)$.
    Hence either $q^*(x)=q'(x)$ or $q^*(x)$ is the unique point found by reflecting through the affine span of $q(w_1),\ldots,q(w_d)$.
    It now follows that
    \begin{equation*}
        r_d(G') = r_d(G',p') \leq 2 r_d(G,p) \leq 2 r_d(G)
    \end{equation*}
    which concludes the proof.
\end{proof}


\subsection{Vertex splitting and spider-splitting}

\subsubsection{Lower bounds for complex realisations}


In this subsection we prove that the realisation number is preserved, and possibly even doubled, under two other graph operations:
vertex splitting and spider-splitting.
Before we do so, we require some technical lemmas regarding realisation numbers.

The proof of the following lemma requires the notion of \emph{multiplicity} of an isolated point.
For the formal definition, we refer to~\cite[Page 224]{Sommese+Wampler}.
However, we will only need the fact that an isolated solution $p \in \CC^n$ of $n$ polynomial equations has multiplicity one if the Jacobian is rank $n$, and multiplicity greater than one otherwise.

\begin{lemma}\label{lem:dominantmapfiber}
    Let $f \colon \CC^n \rightarrow \CC^n$ be a dominant polynomial map, i.e. $f(x) = (f_1(x), \dots, f_n(x))$ for $n$-variate polynomials $f_1, \dots, f_n$.
    For a point $\mu \in \CC^n$, define
    \begin{align*}
    S_\mu &:= \left\{ p \in f^{-1}(\mu) : p \text{ is an isolated point} \right\} \, \mbox{and} \\
    S'_\mu &:= \left\{ p \in S_\mu : \rank d f(p) < n  \right\} \, .
    \end{align*}
    Then for a general point $\lambda \in \CC^n$, we have
    \[
    |f^{-1}(\lambda)| \geq |S_\mu| + |S'_\mu| \, .
    \]
\end{lemma}

\begin{proof}
Write $N_i(\mu)$ for the isolated points of $f^{-1}(\mu)$ of multiplicity $i$.
As $f$ is dominant, it follows from \Cref{lem:dominant} that $f^{-1}(\lambda)$ is zero dimensional for general $\lambda \in \CC^n$, and that each point of $f^{-1}(\lambda)$ is non-singular.
Hence $|f^{-1}(\lambda)| = \sum_{i \geq 1} i\cdot N_i(\lambda) = N_1(\lambda)$.
By~\cite[Theorem 7.1.6]{Sommese+Wampler}, we have 
\begin{equation*}
|f^{-1}(\lambda)| = \sum_{i \geq 1} i\cdot N_i(\lambda)
 \geq \sum_{i \geq 1} i\cdot N_i(\mu) 
 \geq \sum_{i \geq 1} N_i(\mu) + \sum_{i \geq 2} N_i(\mu) 
 = |S_\mu| + |S'_\mu| \, . \qedhere
\end{equation*}
\end{proof}

For a realisation $p \in X_{G,d}$,
we define
\begin{align*}
    \rig (G,p) &:= \left\{ q \in X_{G,d} : f_{G,d}(p) = f_{G,d}(q), \text{ and } (G,q) \text{ is rigid} \right\}, \\
    \fig (G,p) &:= \left\{ q \in \rig (G,p) : (G,q) \text{ is not infinitesimally rigid}, \{q(v_i)\}_{i=1}^d \text{ are affinely independent} \right\}.
\end{align*}
Note that $\fig(G,p)$ is the set of rigid realisations equivalent to $p$ that have infinitesimal flexes.

\begin{lemma}\label{lem:rigsing}
    Let $G$ be minimally $d$-rigid, and take $p \in X_{G,d}$.
    Then
    \begin{equation*}
        2^d c_d(G) \geq |\rig (G,p)| + |\fig (G,p)|.
    \end{equation*}
\end{lemma}

\begin{proof}
    Choose $\lambda, \mu \in \mathbb{C}^E$ such that $\lambda$ is a general point and $f_{G,d}(p) = \mu$.
    Recall $S_\mu$ and $S_\mu'$ from \Cref{lem:dominantmapfiber}, where $f:=\tilde{f}_{G,d}$.
    It follows from the definition of $\tilde{f}_{G,d}$ that $S_\mu = \rig(G,p)$.
    For each $q \in \fig(G,p)$, \Cref{lem:pinnedrigmatrix} implies that $d\tilde{f}_{G,d}(q)$ is not full rank, and hence $q \in S_\mu'$.
    It follows from \Cref{lem:dominantmapfiber} that
    \begin{equation*}
        |\rig (G,p)| + |\fig (G,p)| \leq |S_\mu| + |S'_\mu| \leq \left|\tilde{f}_{G,d}^{-1}(\lambda) \right| \, .
    \end{equation*}
    The result now follows from \Cref{l:xgd}.
\end{proof}

Recall the \emph{neighbourhood} of a vertex $v \in V(G)$ is the set of adjacent vertices $\{u \colon uv \in E(G)\}$.
A \emph{$d$-dimensional vertex split} takes as input a graph $G$, a vertex $x \in V(G)$ and a partition of its neighbourhood into $N_1$, $N_2$ and $\{w_1, \dots, w_{d-1}\}$.
The output is the graph $G'$ constructed from $G - x$ by adding two new vertices $x_1, x_2$ joined to $N_1, N_2$ respectively, along with the edges $\{x_iw_j \colon  i =1,2 \, , \, 1 \leq j \leq d-1\}$ and the edge $x_1x_2$.
A schematic of a $3$-dimensional vertex split is given in \Cref{fig:vsplit}.
Vertex splitting preserves rigidity; see \cite[Corollary 11]{Wsplit}.

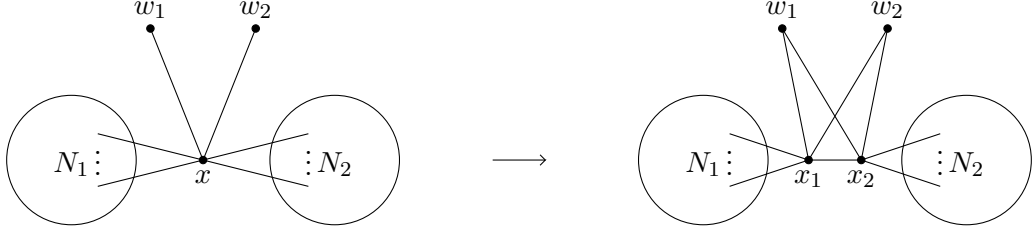
\begin{figure}[t]
\begin{center}
\begin{tikzpicture}[scale=0.7]
\draw (-3,-5) circle (35pt);
\draw (4,-5) circle (35pt);
\draw (9,-5) circle (35pt);
\draw (-8,-5) circle (35pt);

\draw (-3,-5.5) circle (0pt) node[anchor=south]{$N_2$};
\draw (4,-5.5) circle (0pt) node[anchor=south]{$N_1$};
\draw (9,-5.5) circle (0pt) node[anchor=south]{$N_2$};
\draw (-8,-5.5) circle (0pt) node[anchor=south]{$N_1$};

\filldraw (-5.5,-5) circle (2pt) node[anchor=north]{$x$};
\filldraw (6,-5) circle (2pt) node[anchor=north]{$x_1$};
\filldraw (7,-5) circle (2pt) node[anchor=north]{$x_2$};

\filldraw (-6.5,-2.5) circle (2pt) node[anchor=south]{$w_1$};
\filldraw (-4.5,-2.5) circle (2pt) node[anchor=south]{$w_2$};

\filldraw (5.5,-2.5) circle (2pt) node[anchor=south]{$w_1$};
\filldraw (7.5,-2.5) circle (2pt) node[anchor=south]{$w_2$};

\node at (4.5,-4.9){$\vdots$};
\node at (8.5,-4.9){$\vdots$};

\node at (-3.5,-4.9){$\vdots$};
\node at (-7.5,-4.9){$\vdots$};

\draw[black]
(-5.5,-5) -- (-6.5,-2.5);

\draw[black]
(-5.5,-5) -- (-4.5,-2.5);

\draw[black]
(-5.5,-5) -- (-3.5,-4.5);

\draw[black]
(-5.5,-5) -- (-3.5,-5.5);

\draw[black]
(-5.5,-5) -- (-7.5,-5.5);

\draw[black]
(-5.5,-5) -- (-7.5,-4.5);

\draw[black]
(6,-5) -- (7,-5);

\draw[black]
(6,-5) -- (5.5,-2.5);

\draw[black]
(6,-5) -- (7.5,-2.5);

\draw[black]
(5.5,-2.5) -- (7,-5);

\draw[black]
(7.5,-2.5) -- (7,-5);

\draw[black]
(4.5,-4.5) -- (6,-5);

\draw[black]
(4.5,-5.5) -- (6,-5);

\draw[black]
(7,-5) -- (8.5,-4.5);

\draw[black]
(7,-5) -- (8.5,-5.5);

\draw[black]
(0,-5) -- (1,-5) -- (0.9,-5.1);

\draw[black]
(1,-5) -- (0.9,-4.9);

\end{tikzpicture}
\caption{A schematic of the 3-dimensional vertex split.} \label{fig:vsplit}
\end{center}
\end{figure}

\begin{lemma}\label{lem:isolatedvertexsplit}
    Let $(G,p)$ be a generic minimally $d$-rigid framework and let $G'=(V',E')$ be formed from $G$ by a $d$-dimensional vertex split.
    Let $p'$ be the realisation of $G'$ where $p'(v) = p(v)$ for all $v \in V - x$ and $p'(x_1) = p'(x_2) = p(x)$.
    Then $(G',p')$ is rigid but not infinitesimally rigid.
\end{lemma}

\begin{proof}
    We first observe that $(G',p')$ has an edge of length zero, namely $x_1x_2$, and so cannot be infinitesimally rigid.
    To prove rigidity, we observe the following.
    By identifying $x_1$ with $x$, we can assume that $V' = V \cup \{x_2\}$.
    If $(G',q')$ is equivalent to $(G',p')$ with $q'(x_1) = q'(x_2)$,
    then the framework $(G,q)$ with realisation $q = q'|_{V}$ is equivalent to $(G,p)$.
    As we want to prove that $(G',p')$ is rigid,
    we only need to show that for sufficiently close realisations,
    the vertices $x_1,x_2$ lie at the same point.
    We measure the concept of `sufficiently close' using the complex Euclidean norm metric on $(\mathbb{C}^d)^V$,
    here denoted by $m( \, \cdot \, , \cdot \, )$.

    Choose any framework $(G',q')$ equivalent to $(G',p')$.
    By applying translation,
    we may suppose that both $p'(x_1) = \mathbf{0}$ and $q'(x_1) = \mathbf{0}$.
    With this, we have
    \begin{equation}\label{eq0:isolatedvertexsplit}
        \|q'(x_2)\|^2 = \|q'(x_2) - q'(x_1)\|^2 = \|p'(x_2) - p'(x_1)\|^2 = 0.
    \end{equation}
    Now choose $1 \leq i \leq d-1$.
    Since $(G',q')$ is equivalent to $(G',p')$,
    we have
    \begin{equation}\label{eq1:isolatedvertexsplit}
        \|q'(x_2) - q'(w_i)\|^2 = \|p'(x_2) - p'(w_i)\|^2 = \|p'(x_1) - p'(w_i)\|^2 = \|q'(x_1) - q'(w_i)\|^2 = \| q'(w_i)\|^2.
    \end{equation}
    Next,
    we note the following reformulation holds:
    \begin{equation}\label{eq2:isolatedvertexsplit}
        \|q'(x_2) - q'(w_i)\|^2 = \|q'(x_2)\|^2 - 2 q'(x_2) \cdot q'(w_i) + \|q'(w_i)\|^2 = \|q'(w_i)\|^2 - 2 q'(x_2) \cdot q'(w_i).
    \end{equation}
    By combining \eqref{eq1:isolatedvertexsplit} and \eqref{eq2:isolatedvertexsplit},
    we see that
    \begin{equation*}
        \|q'(w_i)\|^2 - 2 q'(x_2) \cdot q'(w_i) = \| q'(w_i)\|^2 \quad \implies \quad q'(x_2) \cdot q'(w_i) = 0.
    \end{equation*}
    As $p'$ is generic, the points $\{p'(w_i)\}_{i=1}^{d-1}$ are contained in a unique linear hyperplane.
    Hence, there exists $\varepsilon > 0$ such that for any $q'$ satisfying $m(p',q') < \varepsilon$, there exists a unique non-zero vector $z \in \mathbb{C}^d$ (up to scaling by $\pm 1$) where $z \cdot q'(w_i) = 0$ for each $1 \leq i \leq d-1$,
    and $\|z\|^2 = 1$.
    In such a situation,
    we then have that $q'(x_2) = \lambda z$ for some $\lambda \in \mathbb{C} \setminus \{0\}$.
    When we substitute this into \eqref{eq0:isolatedvertexsplit},
    we see that $\lambda = 0$.
    Hence, $q'(x_2) = \mathbf{0} = q'(x_1)$ if $m(p',q') < \varepsilon$.
    This now concludes the proof.
\end{proof}

\begin{theorem}\label{thm:vertexsplit}
    Let $G=(V,E)$ be minimally $d$-rigid with at least $d+1$ vertices and let $G'=(V',E')$ be formed from $G$ by a $d$-dimensional vertex split.
    Then $c_d(G') \geq 2 c_d(G)$.
\end{theorem}

\begin{proof}
    Fix vertices $v_1,\ldots,v_d \in V$ as our pinned vertices to define the spaces $X_{G,d},X_{G',d}$ and the maps $\tilde{f}_{G,d},\tilde{f}_{G',d}$.
    Now choose a quasi-generic realisation $p \in X_{G,d}$.
    For each $q \in \tilde{f}_{G,d}^{-1}(\tilde{f}_{G,d}(p))$,
    let $(G',q')$ be the framework defined by $q'(v) = q(v)$ for all $v \in V$ and $q'(x_2) = q(x_1)$.
    As each $q \in \tilde{f}_{G,d}^{-1}(\tilde{f}_{G,d}(p))$ is quasi-generic (\Cref{lem:rigid+generic+implies+quasi}),
    each resulting $q'$ has the property that $\{q'(v_1),\ldots,q'(v_d)\}$ are affinely independent.
    It now follows from \Cref{lem:isolatedvertexsplit} that for each $q \in \tilde{f}_{G,d}^{-1}(\tilde{f}_{G,d}(p))$,
    the realisation $q'$ is contained  in $\fig (G',q')$.
    By \Cref{lem:rigsing} and \Cref{l:xgd},
    we have
    \begin{equation*}
         2^d c_d(G') \geq |\rig(G',p')| + |\fig (G',p') | \geq 2|\fig (G',p') | \geq 2 \left| \tilde{f}_{G,d}^{-1}(\tilde{f}_{G,d}(p)) \right| = 2 \cdot 2^d c_d(G)
    \end{equation*}
    which implies the desired inequality.
\end{proof}

A \emph{$d$-dimensional spider-split} takes as input a graph $G$, a vertex $x \in V(G)$ and a partition of its neighbourhood into $N_1$, $N_2$ and $\{w_1, \dots, w_{d}\}$.
The output is the graph $G'$ constructed from $G - x$ by adding two new vertices $x_1, x_2$ joined to $N_1, N_2$ respectively, along with the edges $\{x_iw_j \colon  i =1,2 \, , \, 1 \leq j \leq d\}$.
A schematic of a $3$-dimensional vertex split is given in \Cref{fig:ssplit}.
Spider-splitting also preserves rigidity: it was known to follow using a simplified version of the proof for vertex splitting in \cite{Wsplit}.
In particular, it preserves infinitesimal rigidity even when in special position.

\begin{figure}[t]
\begin{center}
\begin{tikzpicture}[scale=0.7]
\draw (-3,-5) circle (35pt);
\draw (4,-5) circle (35pt);
\draw (9,-5) circle (35pt);
\draw (-8,-5) circle (35pt);

\draw (-3,-5.5) circle (0pt) node[anchor=south]{$N_2$};
\draw (4,-5.5) circle (0pt) node[anchor=south]{$N_1$};
\draw (9,-5.5) circle (0pt) node[anchor=south]{$N_2$};
\draw (-8,-5.5) circle (0pt) node[anchor=south]{$N_1$};

\filldraw (-5.5,-5) circle (2pt) node[anchor=north]{$x$};
\filldraw (6,-5) circle (2pt) node[anchor=north]{$x_1$};
\filldraw (7,-5) circle (2pt) node[anchor=north]{$x_2$};

\filldraw (-6.5,-2.5) circle (2pt) node[anchor=south]{$w_1$};
\filldraw (-4.5,-2.5) circle (2pt) node[anchor=south]{$w_3$};
\filldraw (-5.5,-2.5) circle (2pt) node[anchor=south]{$w_2$};

\filldraw (5.5,-2.5) circle (2pt) node[anchor=south]{$w_1$};
\filldraw (7.5,-2.5) circle (2pt) node[anchor=south]{$w_3$};
\filldraw (6.5,-2.5) circle (2pt) node[anchor=south]{$w_2$};

\node at (4.5,-4.9){$\vdots$};
\node at (8.5,-4.9){$\vdots$};

\node at (-3.5,-4.9){$\vdots$};
\node at (-7.5,-4.9){$\vdots$};

\draw[black]
(-5.5,-5) -- (-6.5,-2.5);

\draw[black]
(-5.5,-5) -- (-4.5,-2.5);

\draw[black]
(-5.5,-5) -- (-5.5,-2.5);

\draw[black]
(-5.5,-5) -- (-3.5,-4.5);

\draw[black]
(-5.5,-5) -- (-3.5,-5.5);

\draw[black]
(-5.5,-5) -- (-7.5,-5.5);

\draw[black]
(-5.5,-5) -- (-7.5,-4.5);

\draw[black]
(6,-5) -- (5.5,-2.5);

\draw[black]
(6,-5) -- (7.5,-2.5);

\draw[black]
(5.5,-2.5) -- (7,-5);

\draw[black]
(7.5,-2.5) -- (7,-5);

\draw[black]
(4.5,-4.5) -- (6,-5);

\draw[black]
(6.5,-2.5) -- (7,-5);

\draw[black]
(6.5,-2.5) -- (6,-5);

\draw[black]
(4.5,-5.5) -- (6,-5);

\draw[black]
(7,-5) -- (8.5,-4.5);

\draw[black]
(7,-5) -- (8.5,-5.5);

\draw[black]
(0,-5) -- (1,-5) -- (0.9,-5.1);

\draw[black]
(1,-5) -- (0.9,-4.9);

\end{tikzpicture}
\caption{A schematic of the 3-dimensional spider split.} \label{fig:ssplit}
\end{center}
\end{figure}
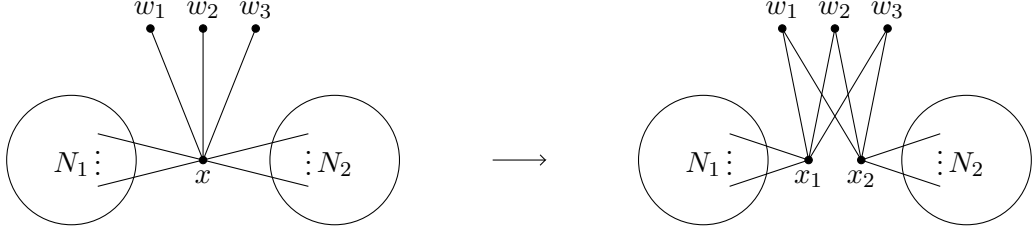

\begin{lemma}\label{lem:folklorespidersplit}
    Let $(G,p)$ be an infinitesimally rigid framework in $\mathbb{C}^d$ and let $G'=(V',E')$ be formed from $G$ by a $d$-dimensional spider-split.
    Let $p'$ to be the realisation of $G'$ where $p'(v) = p(v)$ for all $v \in V-x$ and $p'(x_1) = p'(x_2) = p(x)$.
    Then $(G',p')$ is also infinitesimally rigid.
\end{lemma}

\begin{proof}
    By identifying $x_1$ with $x$, we can assume that $V' = V \cup \{x_2\}$.
    Let $\omega' \in \CC^{E'}$ be a stress of $(G',p')$, i.e. around each vertex $v \in V'$ we have $\sum_{u \in N(v)} \omega'_{vu}(p(v) - p(u)) = 0$.
    By setting $\omega_{x_1v} = \omega'_{x_1v} + \omega'_{x_2v}$ and $\omega_{uv} = \omega'_{uv}$ for all $u,v \neq x_1, x_2$, this reduces to a stress $\omega \in \CC^{E}$ of $(G,p)$.
    We can also obtain another stress $\omega'_G \in \CC^{E}$ of $(G,p)$ by restricting $\omega'$ to $G$.
    As $(G,p)$ is infinitesimally rigid, we have both $\omega = \omega'_G = 0$.
    It follows that $\omega' = 0$, hence $(G',p')$ is also infinitesimally rigid.
\end{proof}

\begin{theorem}\label{thm:spidersplit}
    Let $G=(V,E)$ be minimally $d$-rigid with at least $d+1$ vertices and let $G'=(V',E')$ be formed from $G$ by a $d$-dimensional spider-split.
    Then $c_d(G') \geq c_d(G)$.
\end{theorem}

\begin{proof}
    Fix vertices $v_1,\ldots,v_d \in V$ as our pinned vertices to define the spaces $X_{G,d},X_{G',d}$ and the maps $\tilde{f}_{G,d},\tilde{f}_{G',d}$.
    Now choose a quasi-generic realisation $p \in X_{G,d}$.
    For each $q \in \tilde{f}_{G,d}^{-1}(\tilde{f}_{G,d}(p))$,
    let $(G',q')$ be the framework defined by $q'(v) = q(v)$ for all $v \in V$ and $q'(x_2) = q(x_1)$.
    As each $q \in \tilde{f}_{G,d}^{-1}(\tilde{f}_{G,d}(p))$ is quasi-generic (\Cref{lem:rigid+generic+implies+quasi}),
    each resulting $q'$ has the property that $\{q'(v_1),\ldots,q'(v_d)\}$ are affinely independent.
    It now follows from \Cref{lem:folklorespidersplit} that for each $q \in \tilde{f}_{G,d}^{-1}(\tilde{f}_{G,d}(p))$,
    the realisation $q'$ is contained in $\rig (G',q')$.
    By \Cref{lem:rigsing} and \Cref{l:xgd},
    we have
    \begin{equation*}
         2^d c_d(G') \geq |\rig(G',p')| + |\fig (G',p') | \geq |\rig (G',p') | \geq \left| \tilde{f}_{G,d}^{-1}(\tilde{f}_{G,d}(p)) \right| = 2^d c_d(G)
    \end{equation*}
    which implies the desired inequality.
\end{proof}

\subsubsection{Lower bounds for real realisations}

We can adapt our previous methods to provide lower bounds to the number of real realisations after vertex splitting and spider-splitting.
We first require a lemma bounding the real realisation number via real frameworks in special position.

\begin{lemma}\label{lem:real+lower+bound}
    Let $G$ be minimally $d$-rigid and $v_1, \dots, v_d$ a sequence of $d$ vertices.
    Suppose $(G,p_1),\dots, (G,p_t)$  are infinitesimally rigid real frameworks that are pairwise equivalent but non-congruent, and $p_i(v_1), \dots, p_i(v_d)$ affinely independent for all $1 \leq i \leq t$.
    Then $r_d(G) \geq t$.
\end{lemma}

\begin{proof}
    As $p_1, \dots, p_t \in (\RR^d)^V$, after applying isometries, we can assume that $p_1, \dots, p_t \in X_{G,d}\cap (\RR^d)^V$ where $v_1, \dots, v_d$ are the pinned vertices.
    Consider the pinned rigidity map $\tilde{f}_{G,d}$ restricted to the domain $X_{G,d} \cap (\RR^d)^V$.
    By \Cref{lem:pinnedrigmatrix}, $d\tilde{f}_{G,d}(p_i)$ is invertible for all $1 \leq i \leq t$.
    Applying the inverse function theorem (e.g., \cite[Theorem 9.24]{rudin}), there exists disjoint open sets $U_i$ around $p_i$ and $V$ around $\tilde{f}_{G,d}(p_i)$ such that $f$ maps each $U_i$ diffeomorphically onto $V$.
    As such, we can pick some general point $\lambda \in V$ and there exists unique $q_i \in U_i$ such that $\tilde{f}_{G,d}(q_i) = \lambda$.
    Moreover, as the coordinates of $\lambda$ are algebraically independent, it follows that each $q_i$ is quasi-generic.
    Therefore $r_d(G) \geq r_d(G,q_1) \geq t$.
\end{proof}

\begin{theorem}\label{thm:vertexsplitreal}
    Let $G=(V,E)$ be minimally $d$-rigid with at least $d+1$ vertices and let $G'=(V',E')$ be formed from $G$ by a $d$-dimensional vertex split.
    Then $r_d(G') \geq 2 r_d(G)$.
\end{theorem}

\begin{proof}
    Fix vertices $v_1,\ldots,v_d \in V \setminus \{x\}$ as our pinned vertices to define the spaces $X_{G,d},X_{G',d}$ and the maps $\tilde{f}_{G,d},\tilde{f}_{G',d}$.
    Now choose a quasi-generic real realisation $p \in X_{G,d}$.
    For each real $q \in \tilde{f}_{G,d}^{-1}(\tilde{f}_{G,d}(p))$,
    let $(G',q')$ be the real framework defined by $q'(v) = q(v)$ for all $v \in V \setminus \{x\}$ and $q'(x_1) = q'(x_2) = q(x)$.

    \begin{claim}\label{claim1:vertexsplitreal}
        The left kernel of $R(G',q')$ is $\{t \psi : t \in \mathbb{R}\}$, where $\psi \in \mathbb{R}^{E'}$ satisfies $\psi(x_1x_2)=1$ and $\psi(e) =0$ for all $e \neq x_1x_2$.
    \end{claim}

    \begin{claimproof}
        As $q'(x_1) = q'(x_2)$, it is clear that $\psi \in \ker R(G',q')^{\top}$.
        Choose any $\lambda' \in \ker R(G',q')^{\top}$ and fix $\lambda \in \mathbb{R}^E$ to be the vector where 
        \begin{equation}\label{eq:kernel+cases+vertex+split}
            \lambda (vw) :=
            \begin{cases}
                \lambda' (x_1 w_i) + \lambda' (x_2 w_i) &\text{if } v = x \text{ and } w = w_i, \text{ or vice versa},\\
                \lambda' (x_i w) &\text{if } v = x \text{ and } w \in N_i, \text{ or vice versa},\\
                \lambda' (vw) &\text{otherwise}.
            \end{cases}
        \end{equation}
        Then $\lambda \in \ker R(G,q)^{\top}$ also.
        As $(G,q)$ is rigid and $q$ is quasi-generic by \Cref{lem:rigid+generic+implies+quasi}, it is infinitesimally rigid and hence $\lambda  = \mathbf{0}$.
        This implies that $\lambda'(vw) = 0$ in the second and third cases of \eqref{eq:kernel+cases+vertex+split}, and $\lambda'(x_1 w_i) = -\lambda'(x_2w_i)$ for each $i \in [d-1]$.
        As $\lambda' \in \ker R(G',q')^{\top}$ and $q'(x_1) - q'(x_2) = 0$, the following equation holds around vertex $x_1$:
        \begin{equation*}
            \sum_{i=1}^{d-1} \lambda'(x_1 w_i) (q(x) - q(w_i)) = \sum_{i=1}^{d-1} \lambda'(x_1 w_i) (q'(x_1) - q'(w_i)) = \mathbf{0}.
        \end{equation*}
        As the points $q(x),q(w_1),\ldots,q(w_{d-1})$ are affinely independent,
        we thus have $\lambda'(x_1 w_i) = \lambda'(x_2 w_i) = 0$ for each $i \in [d-1]$.
        Hence $\lambda' = \lambda'(x_1x_2) \psi$.
    \end{claimproof}

    Now fix the sets
    \begin{align*}
        S &:= \left\{ q' \in (\mathbb{R}^d)^{V'} \cap X_{G',d} : q \in \tilde{f}_{G,d}^{-1}(\tilde{f}_{G,d}(p)) \right\},\\
        C &:= \left\{ \rho \in (\mathbb{R}^d)^{V'} \cap X_{G',d} : \tilde{f}_{G'-x_1x_2,d}(\rho) = \tilde{f}_{G'-x_1x_2,d}(p'), ~ \rank R(G'-x_1x_2,\rho)=|E'|-1\right\}.
    \end{align*}
    As a consequence of the constant rank theorem (e.g., \cite[Theorem 9.32]{rudin}), $C$ is a 1-dimensional smooth manifold.
    By \Cref{claim1:vertexsplitreal}, the set $S$ is a finite set contained in $C$.
    In fact, given the morphism
    \begin{equation*}
        h: C \rightarrow \mathbb{R}, \quad \rho \mapsto \|\rho(x_1) - \rho(x_2)\|^2,
    \end{equation*}
    the zeroes of $h$ are exactly the realisations contained in $S$.
    Since $h$ is a continuous polynomial map, it follows that there exists $\varepsilon >0$ and an injective continuous map for each $q' \in S$
    \begin{equation*}
        (-\varepsilon,\varepsilon) \rightarrow C, \quad t \mapsto q'_t
    \end{equation*}
    satisfying $q'_0 = q'$ and $h(q'_t) = |t|$ for each $t \in (-\varepsilon,\varepsilon)$.

    \begin{claim}
        For each $q' \in S$, there exists $\delta_{q'} >0$ such that $(G',q'_t)$ is minimally infinitesimally rigid whenever $0<|t|<\delta_{q'}$.
    \end{claim}

    \begin{claimproof}
        Assume the claim does not hold.
        Then for all $\delta > 0$ there exists infinitely many $0< |t| < \delta$ such that $\rank (G',q'_t) < |E'|$.
        As the set $S$ is finite and $q(w_1)-q(x), \ldots, q(w_{d-1})-q(x)$ are linearly independent, 
        there exists $\delta > 0$ such that for each $0 < t < \delta$ we have $q'_t(x_1) \neq q'_t(x_2)$ and
        \begin{equation}\label{claimvectors}
            q'_t(w_1)- \frac{1}{2}\left(q'_t(x_1) + q'_t(x_2)\right), \dots , q'_t(w_{d-1})-\frac{1}{2}\left(q'_t(x_1) + q'_t(x_2)\right)
        \end{equation}
        are linearly independent.
        As the points $q'_t(x_1)$ and $q'_t(x_2)$ are equidistant from the $d$ points $q'_t(w_1), \ldots, q'_t(w_{d-1})$ and $\frac{1}{2}\left(q'_t(x_1) + q'_t(x_2)\right)$,
        the vector $q'_t(x_1)- q'_t(x_2)$ is orthogonal to each vector in \eqref{claimvectors} when $0 < t < \delta$.
        By Bolzano-Weierstrass, there exists a decreasing sequence $(t_n)_{n \in \mathbb{N}}$ with $0 < t_n < \delta$ and $t_n \rightarrow 0$ as $n \rightarrow \infty$ such that $(G',q'_{t_n})$ is not infinitesimally rigid for all $n$ and
        \[
        \frac{q'_{t_n}(x_1)-q'_{t_n}(x_2)}{\|q'_{t_n}(x_1)-q'_{t_n}(x_2)\|} \rightarrow z \qquad \text{ as } n \rightarrow \infty \, ,
        \] 
        where  $z \in \mathbb{R}^d$ some vector with $\|z\|=1$.
        As $q'_{t_n}(x_1)- q'_{t_n}(x_2)$ is orthogonal to each vector in \eqref{claimvectors} when $t=t_n$ for all $n \in \mathbb{N}$, the vector $z$ is orthogonal to the vectors $q(w_1)-q(x), \ldots, q(w_{d-1})-q(x)$.
        
        With this, we now fix $M_t$ to be the matrix $R(G',q'_{t})$ where each row $vw$ is multiplied by $\|q'_t(v)-q'_t(w)\|^{-1}$,
        and we fix $M_0$ to be the matrix $R(G',q')$ where each row $vw \neq x_1x_2$ is multiplied by $\|q'(v)-q'(w)\|^{-1}$,
        and the row $x_1x_2$ is replaced by the row
        \begin{equation*}
            [ ~ 0 ~ \cdots ~ 0 ~ \overbrace{z^{\top}}^{x_1} ~ 0 ~ \cdots ~ 0 ~ \overbrace{-z^{\top}}^{x_2} ~  0 ~ \cdots ~ 0 ~ ].
        \end{equation*}
        By construction, $\rank M_{t_n} = \rank R(G',q'_{t_n}) \leq |E'|-1$ for each $n$ and $M_{t_n} \rightarrow M_0$ as $n \rightarrow \infty$.
        We can adapt the argument given in \Cref{claim1:vertexsplitreal} to show that any element $\psi$ in the left kernel of $M_0$ must satisfy $\psi(x_1 w_i) = - \psi(x_2 w_i)$ for each $i \in [d-1]$ and $\psi(vw) = 0$ for all other edges aside from $x_1x_2$.
        Hence, the following equation holds around $x_1$
        \begin{equation*}
            \psi(x_1x_2) z + \sum_{i=1}^{d-1} \psi(x_1 w_i) (q(x) - q(w_i)) = \mathbf{0}.
        \end{equation*}
        As $q(w_1)-q(x), \ldots, q(w_{d-1})-q(x)$ and $z$ are a basis,
        we have that $\psi = \mathbf{0}$ and hence $\rank M_0 = |E'|$.
        However, lower-semicontinuity of matrix rank implies that $\rank M_0 \leq \rank M_{t_n} \leq |E'|-1$ for all sufficiently large $n$, a contradiction.
    \end{claimproof}

    Given $0 < \delta < \min\{\delta_{q'} : q' \in S\}$, the set $\{(G',q'_\delta), (G',q'_{-\delta}) : q' \in S\}$ consists of $2|S|$ minimally infinitesimally rigid frameworks with the same edge lengths.
    Hence $r_d(G') \geq 2|S|$ by \Cref{lem:real+lower+bound}.
    The result now follows as $|S|=r_d(G)$.
\end{proof}

\begin{theorem}\label{thm:spidersplitreal}
    Let $G=(V,E)$ be minimally $d$-rigid with at least $d+1$ vertices and let $G'=(V',E')$ be formed from $G$ by a $d$-dimensional spider-split.
    Then $r_d(G') \geq r_d(G)$.
\end{theorem}

\begin{proof}
    Fix vertices $v_1,\ldots,v_d \in V$ as our pinned vertices to define the spaces $X_{G,d},X_{G',d}$ and the maps $\tilde{f}_{G,d},\tilde{f}_{G',d}$.
    Now choose a real quasi-generic realisation $p \in X_{G,d}$ and fix $S \subset X_{G,d} \cap (\mathbb{R}^d)^V$ to be the set of equivalent but non-congruent real realisations of $G$ in $X_{G,d}$.
    For each $q \in S$,
    let $(G',q')$ be the real framework defined by $q'(v) = q(v)$ for all $v \in V$ and $q'(x_1) = q'(x_2) = q(x)$.
    As each $q \in S$ is quasi-generic (\Cref{lem:rigid+generic+implies+quasi}),
    each resulting $q'$ has the property that $\{q'(w_1),\ldots,q'(w_d)\}$ are affinely independent.
    It now follows from \Cref{lem:folklorespidersplit} that for each $q \in S$,
    the framework $(G',q')$ is infinitesimally rigid.
    The result now follows from \Cref{lem:real+lower+bound}.
\end{proof}

\section{Applications} \label{sec:applications}

\subsection{Lower bounds for triangulated spheres}

Using our bounds on how the realisation number changes under vertex splitting, we are ready to prove the following:

\triangulatedsphere*

\begin{proof}
Steinitz \cite{steinitz} proved that the graph of a triangulated sphere can be reduced to $K_4$ using edge contractions, the inverse operation to vertex split.
As such, $G = (V,E)$ can be obtained from $K_4$ by performing $|V| - 4$ vertex splits.
As $c_3(K_4) = r_3(K_4) = 1$, it follows from \Cref{thm:vertexsplit} and \Cref{thm:vertexsplitreal} that $c_3(G) \geq 2^{|V| - 4}$ and $r_3(G) \geq 2^{|V| - 4}$.
\end{proof}

In \cite{KastisPower2023}, Kastis and Power extended Gluck's result on the minimal 3-rigidity of triangulated spheres to \emph{projective planar} graphs; i.e., any graph that has a topological embedding into the real projective plane.
Specifically, they proved that a projective planar graph is minimally 3-rigid if and only if it is \emph{$(3,6)$-tight}; i.e., $|E|=3|V|-6$ and $|E(H)|\leq 3|V(H)|-6$ for all subgraphs with at least 3 vertices.
We can adapt their constructive proof to extend \Cref{thm:triangulatedsphere} to $(3,6)$-tight projective planar graphs.

\begin{theorem}\label{thm:pp}
    Let $G = (V,E)$ be a $(3,6)$-tight projective planar graph.
    Then $c_3(G) \geq r_3(G) \geq 2^{|V| - 4}$.
\end{theorem}

\begin{proof}
    By \cite[Theorem 5.1]{KastisPower2023},
    $G$ can be reduced to one of the following graphs by edge contractions:
    $K_4$, a graph formed from $K_4$ via 0-extensions, the 1-skeleton of the octahedron, or the graph formed from $K_{3,3}$ by adding a new vertex adjacent to all other vertices (here denoted $H$).
    The first two types of graph satisfy $c_3(G) = 2^{|V| - 4}$ and $r_3(G) = 2^{|V| - 4}$ by \Cref{lem:0ext} and \Cref{prop:0extreal}, and the 1-skeleton of the octahedron satisfies $c_3(G) \geq 2^{|V| - 4}$ and $r_3(G) \geq 2^{|V| - 4}$ by \Cref{thm:triangulatedsphere}.
    Using Bartzos and Legerský's database on 3-realisation numbers for small graphs \cite{bartzosLeg}, we see that $c_3(H) = r_3(H) = 8 = 2^{|V(H)|-4}$.
    The result now follows from \Cref{thm:vertexsplit}.
\end{proof}

\subsection{Rigid subgraph substitution}

We now show how the proof of \Cref{thm:isosubreal} can be adapted to understand how rigid subgraph substitution effects realisation numbers.

Let $G =(V,E)$ be a graph with a $d$-rigid subgraph $H =(W,F)$ on at least $d+1$ vertices.
\emph{Rigid subgraph substitution} replaces $H$ with some other $d$-rigid subgraph $H' = (W',F')$ that satisfies $W \subseteq W'$ and $W' \setminus V = W' \setminus W$.
The output is the graph $G' = (V', E')$ with
\[
V' = (V \setminus W) \cup W' \, , \quad E' = (E \setminus F) \cup F' \, .
\]
Rigid subgraph substitution preserves rigidity; see \cite[Corollary 2.3]{FinbowRossWhiteley2012}.

\begin{corollary}\label{cor:rigidsubgraphsub}
    Let $G$ and $H$ be $d$-rigid graphs with at least $d+1$ vertices such that $H$ is a subgraph of $G$. Further suppose that for some minimally $d$-rigid subgraph $\tilde{H}$ of $H$,
    we have that $G - E(H) + E(\tilde{H})$ is minimally $d$-rigid.
    Let $G'$ be obtained from $G$ via the rigid subgraph substitution of $H$ for some $d$-rigid graph $H'$. 
    Then
    \begin{equation*}
        c_d(G') = \frac{c_d(H')}{c_d(H)} c_d(G).
    \end{equation*}
\end{corollary}

\begin{proof}
    First suppose that $V(H')=V(H)$.
    We observe here that the map $h$ described in \Cref{thm:isosubreal} for the pair $(G,H)$ is the same as the map $h$ given for the pair $(G',H')$.
    Hence,
    by \eqref{eq:subeq} given in \Cref{thm:isosubreal},
    we have
    \begin{equation*}
        c_d(G') = (\deg h) c_d(H') = \frac{c_d(G)}{c_d(H)} c_d(H').
    \end{equation*}
    Now suppose that $V(H') \supsetneq V(H)$.
    We apply $k = |V(H')\setminus V(H)|$ $d$-dimensional 0-extensions to $H$ (and hence also $G$) to get the graphs $H''$ and $G''$ such that $V(H'')=V(H')$ and $V(G'') = V(G')$.
    As $c_d(G'') = 2^k c_d(G)$ and $c_d(H'') = 2^k c_d(H)$ by \Cref{lem:0ext} and \Cref{prop:0extreal},
    it now follows from the previous case that
    \begin{equation*}
        c_d(G') = \frac{c_d(H')}{c_d(H'')} c_d(G'') = \frac{c_d(H')}{2^kc_d(H)} 2^k c_d(G)= \frac{c_d(H')}{c_d(H)} c_d(G).\qedhere
    \end{equation*}    
\end{proof}

\subsection{Special cases of 1-extensions, X-replacements and V-replacements}

Using \Cref{cor:rigidsubgraphsub}, we can describe how the realisation number changes under a special case of \emph{$d$-dimensional 1-extension}.
This operation takes a graph $G = (V,E)$ with edge $v_1v_2 \in E$ and distinct vertices $v_3, \dots, v_{d+1}$, and outputs the graph $G' = (V',E')$ with
\[
V' = V \cup \{v_0\} \, , \quad E' = E \setminus \{v_1v_2\} \cup \{v_0v_i \colon 1 \leq i \leq d+1\} \, .
\]
We show that when $v_1, \dots, v_{d+1}$ form a clique in $G$, this operation doubles the realisation number.
This proves Conjectures 1 and 3 of Grasegger \cite{gras25}, which are the cases where $d=2,3$ respectively.

\begin{corollary}\label{cor:1extspecial}
    Let $G$ be a minimally $d$-rigid graph containing a clique with $d+1$ vertices $v_1,\ldots,v_{d+1}$.
    Let $G'$ be the graph formed from $G$ by adding a new vertex $v_0$ adjacent to $v_1,\ldots,v_{d+1}$ and deleting the edge $v_1v_2$.
    Then $c_d(G') = 2 c_d(G)$.
\end{corollary}

\begin{proof}
    Set $H$ to be the complete graph on the vertex set $\{v_1,\ldots,v_{d+1}\}$ and
    set $H'$ to be the graph formed from the complete graph on the vertex set $\{v_0, v_1,\ldots,v_{d+1}\}$ by removing the edge $v_1v_2$.
    Since $H'$ is formed from $K_{d+1}$ by a single $d$-dimensional 0-extension,
    we have $c_d(H') = 2$ by \Cref{lem:0ext}.
    The result now follows from \Cref{cor:rigidsubgraphsub}.
\end{proof}

Moreover, we can additionally prove \cite[Conjecture 4]{gras25}, which describes how the realisation number of $G$ can be altered under a special type of \emph{$X$-replacement} or \emph{$V$-replacement}.
A graph $G'$ is said to be obtained from another graph $G$ by
\begin{itemize}
\item a \emph{($d$-dimensional) $X$-replacement} if there exists non-adjacent edges $v_1v_2$ and $v_3v_4$ in $G$ such that $G'$ is $G - \{v_1v_2, v_3v_4\}$ plus an additional vertex $v_0$ of degree $d+2$ adjacent to $v_1, \dots, v_{d+2}$.
\item a \emph{($d$-dimensional) $V$-replacement} if there exists adjacent edges $v_1v_2$ and $v_2v_3$ in $G$ such that $G'$ is $G - \{v_1v_2, v_2v_3\}$ plus an additional vertex $v_0$ of degree $d+2$ adjacent to $v_1,\dots,v_{d+2}$.
\end{itemize}
We are particularly interested in the $d=3$ case, where these operations are key candidates for generating all minimally $3$-rigid graphs.
Even in this case, it is known that $V$-replacement does not preserve rigidity and $X$-replacement is only conjectured to.
However, if the induced subgraph of $G$ on $v_1, \dots, v_5$ is minimally $3$-rigid, then rigidity is preserved and the 3-realisation number doubles.

\begin{figure}[tp]
\begin{center}
\begin{tikzpicture}[scale=0.8]

\draw (0,0) circle (30pt);
\draw (4,0) circle (30pt);

\filldraw (4,2) circle (3pt);

\filldraw (.5,-.5) circle (3pt);
\filldraw (.5,.5) circle (3pt);
\filldraw (4.5,-.5) circle (3pt);
\filldraw (4.5,.5) circle (3pt);

\filldraw (-.5,-.5) circle (3pt);
\filldraw (-.5,.5) circle (3pt);
\filldraw (3.5,-.5) circle (3pt);
\filldraw (3.5,.5) circle (3pt);

\filldraw (0,0) circle (3pt);
\filldraw (4,0) circle (3pt);

\draw[black,thick]
(.5,.5) -- (.5,-.5);

\draw[black,thick]
(-.5,.5) -- (-.5,-.5);

\draw[black,thick]
(4.5,-.5) -- (4,2) -- (4.5,.5);

\draw[black,thick]
(3.5,-.5) -- (4,2) -- (4,0);

\draw[black,thick]
(3.5,.5) -- (4,2);

\draw[black]
(1.5,0) -- (2.5,0);

\draw[black,thick]
(2.4,-.2) -- (2.5,0) -- (2.4,.2);


\draw (8,0) circle (30pt);
\draw (12,0) circle (30pt);

\filldraw (12,2) circle (3pt);

\filldraw (8.5,-.5) circle (3pt);
\filldraw (8.5,.5) circle (3pt);
\filldraw (12.5,-.5) circle (3pt);
\filldraw (12.5,.5) circle (3pt);

\filldraw (7.5,-.5) circle (3pt);
\filldraw (7.5,.5) circle (3pt);
\filldraw (11.5,-.5) circle (3pt);
\filldraw (11.5,.5) circle (3pt);

\filldraw (8,0) circle (3pt);
\filldraw (12,0) circle (3pt);

\draw[black,thick]
(8.5,.5) -- (8.5,-.5);

\draw[black,thick]
(7.5,-.5) -- (8.5,-.5);

\draw[black,thick]
(12.5,-.5) -- (12,2) -- (12.5,.5);

\draw[black,thick]
(11.5,-.5) -- (12,2) -- (12,0);

\draw[black,thick]
(11.5,.5) -- (12,2);

\draw[black]
(9.5,0) -- (10.5,0);

\draw[black,thick]
(10.4,-.2) -- (10.5,0) -- (10.4,.2);
\end{tikzpicture}
\caption{Illustration of $X$-replacement and $V$-replacement in three dimensions.}
\label{fig:xvpic}
\end{center}
\end{figure}
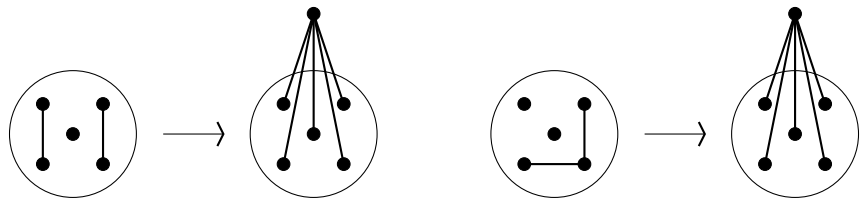

\begin{corollary}\label{cor:xvreplacement}
    Let $G$ be a minimally 3-rigid graph containing a minimally 3-rigid subgraph $H$ with 5 vertices.
    Let $G'$ be a minimally 3-rigid graph formed from $G$ by deleting two edges $e,f \in E(H)$ and adding a new vertex $x$ adjacent to each vertex in $V(H)$.
    Then $c_3(G') = 2c_3(G)$.
\end{corollary}

\begin{proof}
    Fix $H'$ to be the subgraph of $G'$ induced by the vertex set $V(H) \cup \{x\}$.
    By \Cref{cor:rigidsubgraphsub},
    $c_3(G') = c_3(G) \frac{c_3(H')}{c_3(H)}$,
    so it suffices to show that $c_3(H') = 2c_3(H)$.
    There is exactly one minimally 3-rigid graph with 5 vertices, namely $K_5$ minus an edge, and this has a 3-realisation number of 2.
    Moreover, there are exactly three minimally 3-rigid graphs with 6 vertices and a degree 5 vertex as listed in the database \cite{bartzosLeg};
    these are given in \Cref{fig:3rigid6vertex}.
    Since each graph in \Cref{fig:3rigid6vertex} can be constructed from $K_4$ using 3-dimensional 0-extensions, they each have a 3-realisation number of 4.
    Hence $c_3(H') = 2c_3(H)$ as required.
\end{proof}

\begin{figure}
    \centering
    \begin{tikzpicture}[scale=1]
			\node[vertex] (0) at (-2,0) {};
			\node[vertex] (1) at (0,0) {};
			\node[vertex] (2) at (2,0) {};
			
			\node[vertex] (3) at (-1,-1) {};
			\node[vertex] (4) at (1,-1) {};
			\node[vertex] (5) at (0,1) {};

			\draw[edge] (0)edge(3);
			\draw[edge] (0)edge(4);
			\draw[edge] (1)edge(3);
			\draw[edge] (1)edge(4);
            \draw[edge] (2)edge(3);
			\draw[edge] (2)edge(4);
			\draw[edge] (3)edge(4);

            \draw[edge] (0)edge(5);
            \draw[edge] (1)edge(5);
            \draw[edge] (2)edge(5);
            \draw[edge] (3)edge(5);
            \draw[edge] (4)edge(5);
		\end{tikzpicture}
        \qquad
            \begin{tikzpicture}[scale=1]
			\node[vertex] (0) at (-2,0) {};
			\node[vertex] (1) at (0,0) {};
			\node[vertex] (2) at (2,0) {};
			
			\node[vertex] (3) at (-1,-1) {};
			\node[vertex] (4) at (1,-1) {};
			\node[vertex] (5) at (0,1) {};

			\draw[edge] (0)edge(3);
			\draw[edge] (0)edge(1);
			\draw[edge] (1)edge(3);
			\draw[edge] (1)edge(4);
            \draw[edge] (2)edge(1);
			\draw[edge] (2)edge(4);
			\draw[edge] (3)edge(4);

            \draw[edge] (0)edge(5);
            \draw[edge] (1)edge(5);
            \draw[edge] (2)edge(5);
            \draw[edge] (3)edge(5);
            \draw[edge] (4)edge(5);
		\end{tikzpicture}
        \qquad
            \begin{tikzpicture}[scale=1]
			\node[vertex] (0) at (-2,0) {};
			\node[vertex] (1) at (0,0) {};
			\node[vertex] (2) at (2,0) {};
			
			\node[vertex] (3) at (-1,-1) {};
			\node[vertex] (4) at (1,-1) {};
			\node[vertex] (5) at (0,1) {};

			\draw[edge] (0)edge(3);
			\draw[edge] (0)edge(4);
			\draw[edge] (1)edge(0);
			\draw[edge] (1)edge(2);
            \draw[edge] (2)edge(3);
			\draw[edge] (2)edge(4);
			\draw[edge] (3)edge(4);

            \draw[edge] (0)edge(5);
            \draw[edge] (1)edge(5);
            \draw[edge] (2)edge(5);
            \draw[edge] (3)edge(5);
            \draw[edge] (4)edge(5);
		\end{tikzpicture}
    \caption{The three minimally 3-rigid graphs on six vertices with a vertex of degree 5.}
    \label{fig:3rigid6vertex}
\end{figure}
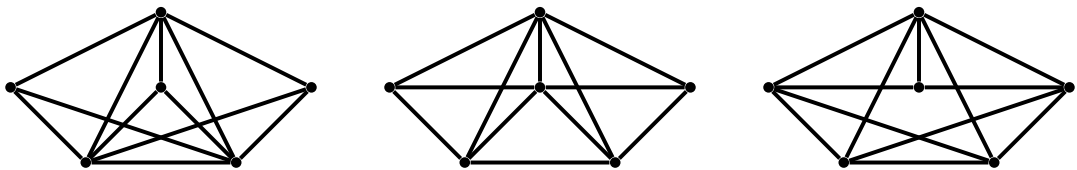

\subsection*{Acknowledgements}

We thank an anonymous referee for careful reading and helpful comments.
S.\,D.\ was supported by the Heilbronn Institute for Mathematical Research,
and partially supported by the KU Leuven grant iBOF/23/064 and the FWO grants G0F5921N (Odysseus) and G023721N.
A.\,N.\ and B.\,S.\ were partially supported by EPSRC grant EP/X036723/1. A.\,N.\ was partially supported by UK Research and Innovation (grant number UKRI1112), under the EPSRC Mathematical Sciences Small Grant scheme.
For the purpose of open access, the authors have applied a Creative Commons Attribution (CC-BY) licence to any Author Accepted Manuscript version arising.

\bibliographystyle{plainurl}
{\small
\bibliography{ref}
}

\appendix

\section{Genericity and quasi-genericity} \label{sec:appendix}

We recall that a realisation $p$ of a graph $G=(V,E)$ is \emph{generic} if $\mathrm{td}[\mathbb{Q}(p): \mathbb{Q}] = d|V|$,
and \emph{quasi-generic} if it is congruent to a generic realisation.

\begin{lemma}[{\cite[Lemma B.1]{dewar23}}]\label{lem:canonical+position}
	Let $G=(V,E)$ be a graph with $|V| \geq d+1$,
	and fix a sequence of $d$ vertices $v_1,\ldots, v_d$.
	For each realisation $p \in (\mathbb{C}^d)^V$,
	define the $(d-1) \times (d-1)$ symmetric matrix
	\begin{align*}
		\mathbb{G}(p) := 		
		\begin{bmatrix}
			(p(v_2) - p(v_1))^{\top} \\
			\vdots \\
			(p(v_d) - p(v_1))^{\top}
		\end{bmatrix}
		\begin{bmatrix}
			p(v_2) - p(v_1) & \cdots & p(v_d) - p(v_1)
		\end{bmatrix}.
	\end{align*}
	Then there exists a realisation $q \in X_{G,d}$ congruent to $p$ if $\mathbb{G}(p)$ only has non-zero leading principal minors.\footnote{A \emph{leading principal minor (of order $n$)} of a square matrix is the determinant of the matrix formed by taking the first $n$ rows and columns.}
\end{lemma}

\begin{lemma}\label{lem:quasixgd}
    Let $(G,p)$ be a quasi-generic rigid framework in $\mathbb{C}^d$ with at least $d+1$ vertices.
    Then for any choice of vertices $v_1,\ldots,v_d$ defining the space $X_{G,d}$,
    there exists a congruent framework $(G,q)$ with $q \in X_{G,d}$.
\end{lemma}

\begin{proof}
    Let $(G,p')$ be a generic realisation that is equivalent to $(G,p)$.
    It follows from $p,p'$ being congruent that $\mathbb{G}(p') = \mathbb{G}(p)$.
    As $p'$ is generic, the matrix $\mathbb{G}(p')$ has only non-zero leading principal minors.
    The result now holds by \Cref{lem:canonical+position}.
\end{proof}

The next lemma follows using the argument given in~\cite[Lemma 3.4]{JO19} since we make the additional assumption that $(G,p)$ is rigid.

\begin{lemma}\label{lem:generic+can+pos}
Let $(G,p)$ be a quasi-generic rigid framework in $\mathbb{C}^d$.
Then $\|p(v)- p(w)\|^2 \neq 0$ for each distinct pair $v,w \in V$, and $\td[\QQ(f_{G,d}(p)) \colon \QQ] = \rank R(G,p)$.
Furthermore, if $p$ is in $X_{G,d}$ then $\overline{\QQ(p)} = \overline{\QQ(f_{G,d}(p))}$.
\end{lemma}

We now have the necessary tools to prove the following lemma from \Cref{sec:prelims}:

\genericimpliesquasi*


\begin{proof}
    \Cref{lem:generic+can+pos} with the assumption that $(G,p)$ is rigid implies that 
    $$\textrm{td}[\QQ(f_{G,d}(p)) \colon \QQ] = \rank R(G,p) = d|V| - \binom{d+1}{2}.$$
    As $f_{G,d}(p) = f_{G,d}(q)$, using \Cref{lem:canonical+position} we can deduce that there exists a framework $(G,q^*)$ in $X_{G,d}$ that is congruent to $(G,q)$.
    Moreover, \Cref{lem:generic+can+pos} gives 
    \[
    \textrm{td}[\QQ(q^*)\colon\QQ] = \textrm{td}[\QQ(f_{G,d}(q^*))\colon\QQ] = \textrm{td}[\QQ(f_{G,d}(p)) \colon \QQ] = d|V| - \binom{d+1}{2} \, ,
    \]
    and hence some set $I$ of $d|V| - \binom{d+1}{2}$ coordinates of $q^*$ are algebraically independent.
    Now choose $A \in O(d,\mathbb{C})$ and $x \in \mathbb{C}^d$ so that the combined set of entries from the upper left triangle of $A$ and the coordinates of $x$ are algebraically independent over $\QQ(q^*)$.
    When we apply the affine transformation of $(A,x)$ to $q^*$,
    we get a generic framework $(G,\tilde{q})$. 
    As $(G,q)$ is congruent to $(G,\tilde{q})$, it is quasi-generic.
\end{proof}

Recall that for us, a generic point $p \in \CC^n$ is a point whose coordinates are all algebraically independent over $\QQ$.
To prove \Cref{thm:generic-implies-general}, we utilise a standard result on generic points in $\CC^n$ under polynomial maps with rational coefficients: see \cite[Lemma 3.2(c)]{JO19} for a proof.

\begin{lemma}\label{lem:generic+fibre+cardinality}
Let $f \colon \CC^n \rightarrow \CC^m$ by $f(p) = (f_1(p), \dots, f_m(p))$  where $f_i \in \QQ[X_1, \dots, X_m]$ for $1 \leq i \leq m$, and let $W(p) = \{q \in \CC^n \; \colon \; f(p) = f(q)\}$ for each $p \in \CC^n$.
If $p$ is generic and $df(p)$ has rank $n$, then $W(p)$ is finite and $|W(p)| = |W(q)|$ for all generic $q \in \CC^n$.
\end{lemma}

\genericimpliesgeneral*

\begin{proof}
The result is trivially true if $|V| \leq d$, so assume $|V| \geq d+1$.
By \Cref{lem:quasixgd}, there exists some realisation $q \in X_{G,d}$ such that $(G,p)$ is congruent to $(G,q)$.
Recall that by the definition of the realisation space, we have $|C_d(G,p)| = |C_d(G,q)|$.
Recall that $X_{G,d} \cong \CC^{d|V| - \binom{d+1}{2}}$ as $\binom{d+1}{2}$ coordinates of $q$ have been pinned to zero.
By the same application of \Cref{lem:generic+can+pos} as in the proof of \Cref{lem:rigid+generic+implies+quasi}, we can deduce that $\td[\QQ(q) \colon \QQ] = d|V| - \binom{d+1}{2}$, and hence all the non-zero coordinates of $q$ are algebraically independent over $\QQ$.

We next apply \Cref{lem:generic+fibre+cardinality} to the pinned rigidity map $\tilde{f}_{G,d} \colon X_{G,d} \rightarrow \CC^{E}$.
The previous paragraph showed that $q$ is a generic point of $X_{G,d} \cong \CC^{d|V| - \binom{d+1}{2}}$.
Moreover, $(G,q)$ is infinitesimally rigid, hence \Cref{lem:pinnedrigmatrix} implies $d\tilde{f}_{G,d}(p)$ has rank $d|V|-\binom{d+1}{2}$.
Therefore \Cref{lem:generic+fibre+cardinality} 
implies that $|C(G,p)|$ is constant for all quasi-generic realisations of $G$.

Finally, let $U$ be the non-empty Zariski open subset of $(\CC^d)^V$ such that $|C_d(G,q)| = c_d(G)$ for all $q \in U$.
As there is no non-trivial polynomial relation that all generic points satisfy, at least one generic realisation $p$ is contained in $U$.
Finally, since the cardinality of $|C(G,p)|$ is constant for all quasi-generic realisations of $G$, we have $|C(G,p)|=c_d(G)$ for all quasi-generic $p$.
\end{proof}

\end{document}